\documentclass[twoside,12pt]{article}
\usepackage[usenames,dvipsnames]{color}
\usepackage{amsmath}
\usepackage{amsthm}
\usepackage{amssymb,bbm}
\usepackage[mathcal]{euscript}
\usepackage{cite}
\usepackage{hyperref}
\usepackage{enumitem}
\usepackage{tikz}
\usepackage{mathtools}

\definecolor{rosso}{rgb}{0.85,0,0}
\definecolor{azzurro}{rgb}{0.13, 0.67, 0.8}

\def\an #1{{\color{rosso}#1}}
\def\an #1{#1}

\def\luca #1{{\color{rosso}#1}}
\def\luca #1{#1}
\def\maunew #1{{\color{blue}#1}}
\def\mau #1{#1}
\def\maunew #1{#1}
\def\rev #1{{\color{rosso}#1}}
\def\rev #1{#1}
\def\sub #1{{\color{rosso}#1}}
\def\sub #1{#1}

\def\mod #1{{\color{rosso}#1}}
\def\anlast #1{{\color{rosso}#1}}
\def\anlastr #1{{\color{rosso}#1}}
\def\mod #1{{#1}}
\def\anlast #1{{#1}}
\def\anlastr #1{{#1}}

\newtheorem{theorem}{Theorem}[section]
\newtheorem{remark}[theorem]{Remark}

\newtheorem{proposition}[theorem]{Proposition}
\newtheorem{lemma}[theorem]{Lemma}

\numberwithin{equation}{section}

\let\non\nonumber

\def\Lip{Lip\-schitz}

\def\lhs{left-hand side}
\def\rhs{right-hand side}

\def\multibold #1{\def\arg{#1}%
  \ifx\arg\pto \let\next\relax
  \else
  \def\next{\expandafter
    \def\csname #1#1#1\endcsname{{\bf #1}}%
    \multibold}%
  \fi \next}

\def\pto{.}

\def\multical #1{\def\arg{#1}%
  \ifx\arg\pto \let\next\relax
  \else
  \def\next{\expandafter
    \def\csname cal#1\endcsname{{\cal #1}}%
    \multical}%
  \fi \next}

\def\multimathop #1 {\def\arg{#1}%
  \ifx\arg\pto \let\next\relax
  \else
  \def\next{\expandafter
    \def\csname #1\endcsname{\mathop{\rm #1}\nolimits}%
    \multimathop}%
  \fi \next}

\multibold
qwertyuiopasdfghjklzxcvbnmQWERTYUIOPASDFGHJKLZXCVBNM.

\multical
QWERTYUIOPASDFGHJKLZXCVBNM.

\multimathop
diag dist div dom mean meas sign supp .

\def\<#1>{\mathopen\langle #1\mathclose\rangle}
\def\norma #1{\left\| #1\right\|}
\newcommand{\vertiii}[1]{{\left\vert\kern-0.25ex\left\vert\kern-0.25ex\left\vert #1 \right\vert\kern-0.25ex\right\vert\kern-0.25ex\right\vert}}

\def\I2 #1{\int_{Q_t}|{#1}|^2}
\def\IT2 #1{\int_{Q_t^T}|{#1}|^2}
\def\IO2 #1{\norma{{#1(t)}}^2}
\def\IOT2 #1{\norma{{#1(T)}}^2}
\def\ov #1{{\overline{#1}}}

\def\iot {\int_0^t}

\def\iO{\int_\Omega}

\def\dt{\partial_t}

\def\dn{\partial_{\bf n}}
\def\checkmmode #1{\relax\ifmmode\hbox{#1}\else{#1}\fi}

\def\erre{{\mathbb{R}}}

\def\Vp{{V^*}}

\def\genspazio #1#2#3#4#5{#1^{#2}(#5,#4;#3)}
\def\spazio #1#2#3{\genspazio {#1}{#2}{#3}T0}

\def\spazios #1#2#3{\genspazio {#1}{#2}{#3}T \sigma}
\def\L {\spazio L}
\def\H {\spazio H}
\def\C #1#2{C^{#1}([0,T];#2)}

\def\Ls {\spazios L}

\def\Lx #1{L^{#1}(\Omega)}
\def\Hx #1{H^{#1}(\Omega)}

\def\b{\beta}	
\def\d{\delta}

\def\s{\sigma}
\def\m{\mu}	
\def\ph{\varphi}

\def\cd{C_{\d}}
\let\lacc\l
\def\l{{\lambda}}
\def\phl{{\ph_\lambda}}
\def\ml{{\m_\lambda}}

\newcommand{\NN}{{\cal N}}
\newcommand{\RR}{{\cal R}}

\newcommand{\ds}{{\rm ds}}

\newcommand\dd {{\delta^h_t}}

\newcommand\emb {{\hookrightarrow}}

\renewcommand{\hat}{\widehat}

\def\Accorpa #1#2 #3 {\gdef #1{\eqref{#2}--\eqref{#3}}%
  \wlog{}\wlog{\string #1 -> #2 - #3}\wlog{}}

\def\dd#1 {{\rm d#1}}
\def\h{\mathbbm{h}}
\def\tre{{{\gamma}}}

\begin{document}

\title{Cahn--Hilliard equations with singular potential, \\ reaction term and pure phase initial datum}
\author{Maurizio Grasselli \footnotemark[1] \and Luca Scarpa \footnotemark[1] \and Andrea Signori \footnotemark[2]}
\date{}
\maketitle

\renewcommand{\thefootnote}{\fnsymbol{footnote}}
\footnotetext[1]{Dipartimento di Ma\-te\-ma\-ti\-ca, Politecnico di Milano, 20133 Milano, Italy ({\tt mau\-ri\-zio.gras\-sel\-li@polimi.it, luca.scarpa@polimi.it}),}
\footnotetext[2]{Dipartimento di Ma\-te\-ma\-ti\-ca, Politecnico di Milano, 20133 Milano, Italy, Alexander von Humboldt Research Fellow ({andrea.signori@polimi.it}).}

\begin{abstract}\noindent
We consider local and nonlocal Cahn--Hilliard equations with constant mobility and singular potentials including, e.g., the Flory--Huggins potential, subject to no-flux (or periodic) boundary conditions. The main goal is to show that the
presence of a suitable class of reaction terms allows to establish the existence of a weak solution to the corresponding initial and boundary value problem even though
the initial condition is a pure state.
This fact was already observed by the authors in a previous contribution devoted to a specific biological model.
In this context, we examine the essential assumptions required for the reaction term to ensure the existence of a weak solution.
Also, we explore the scenario involving the nonlocal Cahn--Hilliard equation and provide some illustrative examples that contextualize within our abstract framework.
\end{abstract}


\vspace{2mm}


\vspace{2mm}


\section{Introduction}
\label{SEC:INTRO}
Phase separation is a relevant phenomenon which appears in a number of different contexts (see, for instance, \cite{DG,PFSSSD} and references therein).
In particular, it has recently become a paradigm in Cell Biology (see, e.g., \cite{Boe,Dol,HWJ}).
Phase separation takes place whenever two (or more) species interact in a way that can be described as a competition between mixing entropy and a demixing effect due to the attraction of particles
of the same species. As a prototypical example, one may think of oil and vinegar. From the mathematical viewpoint, a fairly good description of such phenomenon is given by the so-called Cahn--Hilliard equation (see \cite{CH1,CH2}, see also \cite{Mbook} and its references). In the case of two species, this equation assumes the following form
\begin{equation*}
\dt \ph = \div(m(\ph)\nabla \mu)
\end{equation*}
in a given bounded domain $\Omega \subset \erre^d$, $d \in \{1,2,3\}$ and in a time interval $(0,T)$ for some given $T>0$. Here, $\ph$ represents the (relative) difference between
the concentration of the two species, $m(\cdot)$ is a nonnegative mobility function, and $\mu$ is the so-called chemical potential. The latter corresponds to the functional derivative of the (local) free energy
\begin{align}	
	\label{energy:local}
	{\cal E}^{\rm loc} (\ph(\cdot))
	= \iO |\nabla \ph (\cdot,x)|^2 \dd x
	+ \iO F(\ph(\cdot,x)) \, \dd x,
\end{align}
where a physically relevant example of potential density $F$ is given by the logarithmic (\an{a.k.a.} the Flory\an{--}Huggins) potential (see, e.g., \cite{CMZ} and references therein)
\begin{equation}
  \label{Flog}
  F_{\rm log}(r):=\frac{\vartheta}{2}\left[(1+r)\ln(1+r)+(1-r)\ln(1-r)\right]-\frac{\vartheta_0}{2}r^2,
  \quad r\in(-1,1),
\end{equation}
where  $0<\vartheta<\vartheta_0$.
An alternative choice is a nonlocal free energy of the form (see \cite{GL1,GL2}, see also \cite{AH,GGG1,GGG2,DRST,DST1,DST2} and their references)
\begin{align}	
	\label{energy:nonlocal}
	{\cal E}^{\rm nonloc} (\ph(\cdot))
	=
\frac{1}{4}\int_{\Omega \times\Omega} K(x-y) \vert \varphi (\cdot,x) - \varphi(\cdot,y)\vert^2{\dd x \dd y }
	+ \iO F(\ph(\cdot,x)) \,\dd x,
\end{align}
where $K: \mathbb{R}^d \to \mathbb{R}$ is a sufficiently smooth interaction kernel such that $K(x)=K(-x)$.

Taking constant mobility equal to one, we have the following local Cahn--Hilliard equation
\begin{equation}
\label{LCHE}
\dt \ph = \Delta (- \Delta \ph + F^\prime(\ph)),
\end{equation}
and its nonlocal counterpart
\begin{equation}
\label{NLCHE}
\dt \ph = \Delta (a\ph -K* \ph + F'(\ph)),
\end{equation}
in $Q:=\Omega \times (0,T)$, where
\[
  (K*v)(x):=\int_\Omega K(x-y)v(y)\, \dd y, \quad x\in\Omega\an{,}
\]
for any given measurable function $v:\Omega\to\erre$ and $a=K*1$.

\an{To} ensure mass conservation, the above equations are usually equipped with no-flux boundary conditions on $\Sigma:=\partial\Omega\times(0,T)$ \an{for the chemical potential $\mu$},
even though periodic boundary conditions can also be considered (see, for instance, \cite{GL2}).
In the vast mathematical literature devoted to these equations, in order to prove the existence of a reasonable (say, weak) solution in the case of a singular potential like \eqref{Flog},
it is always assumed that the spatial average of the initial datum must belong to $(-1,1)$ (see, for instance, \cite[Thm.~2.2]{Dett} and \cite[Def.~6.1, Thm.~6.1]{KNP}).
However, it must be noted that, in the case of degenerate mobility, pure phases are allowed to be solutions by a suitable splitting of the chemical potential
(see \cite{EG}, see also \cite{FGG} and references therein for the nonlocal case). In the case of equation \eqref{LCHE}, this fact is less obvious (see \cite[Rem.~4.2]{MZ}).


Here, inspired by a biological model on the formation of condensates (see \cite{GFGN,GFGNrev}
and references therein, see also \cite{Lahaetal}) which has been analyzed in \cite{GSS},
we show that it is indeed possible to establish the existence of a weak solution to
Cahn--Hilliard equations \eqref{LCHE} and \eqref{NLCHE} with $F$ like, e.g., \eqref{Flog}, equipped with no-flux boundary conditions, and subject to a suitable reaction term,
if the initial datum is a pure phase.

More precisely, letting $d\in\{2,3\}$, we consider the following initial and boundary value problems
\begin{alignat}{2}	
	\label{sys:1}
	&  \dt \ph - \Delta \mu  = S(x,\ph,g) \qquad && \text{in $Q$,}
	 \\
	\label{sys:2}
	 & \mu  =  -\Delta \ph + F'(\ph) + g  \qquad && \text{in $Q$,}
	 \\
	\label{sys:3}
	&  \dn \ph
	 =  \dn\mu
	= 0  \qquad && \text{on $\Sigma$,}
	 \\	
	 \label{sys:4}
	&  \ph(0) = -1  \qquad && \text{in $\Omega$,}
\end{alignat}
and
\begin{alignat}{2}	
	\label{sys:non:1}
	&  \dt \ph - \Delta \mu  = S(x,\ph,g) \qquad && \text{in $Q$,}
	 \\
	\label{sys:non:2}
	 & \mu  =  a\ph -K* \ph + F'(\ph) + g  \qquad && \text{in $Q$,}
	 \\
	\label{sys:non:3}
	&   \dn \m
	= 0  \qquad && \text{on $\Sigma$,}
	 \\	
	 \label{sys:non:4}
	&  \ph(0) = -1  \qquad && \text{in $\Omega$,}
\end{alignat}
where $S(x,\ph,g)$ is a reaction term and $g$ is a given source, possibly depending on time.
We incidentally notice that in the above system we abuse a bit of notation to denote the function $S(x,\ph,g)$ by explicitly writing the space variable $x \in \Omega$ which is not specified for all the other variables. As no confusion may arise, that convention will be in force throughout the paper.

The main technical difficulty of the entire work derives from the fact that
we can no longer prove \an{the usual property} that $\mu$ belongs to $L^2(0,T;L^2(\Omega))$, like in \cite{Lam,Mir}
(see also \cite{IM,MR} for the nonlocal case with degenerate mobility),
through an inequality like \cite[Prop.~A.2]{MZ}
(see also \cite[Sec.~5]{GMS09}). This is due to the fact that \an{in our scenario,} we start with a pure phase. However, here the mass is not conserved
because of the reaction term. Thus, under appropriate assumptions, we can show
that the spatial average of $\varphi$ instantly detaches from $-1$ (i.e.\rev{,} the phase becomes
mixed) and to explicitly compute the behavior close to $-1$ for small times.
This is based on {\it ad hoc} qualitative analysis of evolution of the spatial average of $\ph$ \an{given by}
$$
(\ph(t))_\Omega := \frac 1{\vert \Omega\vert} \int_\Omega \ph(x,t)\,\dd x,
$$
where $\vert\Omega\vert$ denotes the $d$\rev{-}Lebesgue measure of $\Omega$, which obeys the Cauchy problem
\begin{equation}
\label{ODEmean}
	y' (t)= (S(\cdot,\ph(t),g(t)))_\Omega \quad \text{for all $t \in [0,T],$}\quad y(0)=-1,
\end{equation}
where $y(t)\rev{:}=(\ph(t))_\Omega$.
To obtain the above relation, we integrate equation \eqref{sys:1} (\eqref{sys:non:1} resp.) over the domain $\Omega$ and then
divide by $\vert\Omega\vert$.

The second step consists in generalizing the above mentioned inequality by explicitly tracking
the behavior of $F^\prime(\ph)$ close to the pure phase. As we shall see, this argument leads to a
local in time $L^p$-control in time of $\mu$ for $p\in (1,2)$ which is enough to recover the existence of a weak solution.
The fact that we start from $-1$ is actually just a matter of choice. We could also start from $1$ by just  taking $-S$ in place of $S$.
This essentially means that we are removing mass instead of adding it.

The analysis of \eqref{ODEmean} imposes some restrictions on $S$ (see below). Nonetheless, we can handle, for instance, two well-known types of Cahn--Hilliard
equations with reaction terms, namely,
a Cahn--Hilliard--Oono type equation (see \cite{GDM}, see also \cite{GGM} for theoretical results)
and the inpainting model proposed by Bertozzi et al. (see \cite{BEG1,BEG2}, see also \cite{CFM} for the case of logarithmic potential and \cite{Jiangetal}
for the nonlocal case).
The former can be deduced by taking $g\equiv0$ and setting
\begin{equation}
\label{CHO}
S(x,\ph,g)=S(\ph)=-m(\ph - c)\an{,}
\end{equation}
where $m>0$ and $c\in (-1,1)$ are given. With reference to the (local) model
we recall that $c$ is related to the backward and forward reaction rates of the chemical reaction (cf. \cite{GDM}).
Observe that this is an off-critical case and the mass is not conserved.
Regarding the latter, we have $g\equiv0$ and
\begin{equation}
\label{Bert}
S(x,\ph,g)=S(x,\ph)=-\lambda_0\chi_U(x)(\ph - c(x))
\end{equation}
where $\lambda_0>0$, $\chi_U$ is the characteristic function of a given measurable set $U$ (the complement of the inpainting region)
such that $0<\vert U\vert<\vert \Omega\vert$ and
$c: \Omega \to [-1,1]$ is a given bounded function (i.e., the given image). In this case, taking $-1$ as initial datum means that one
can start with a totally white image. Then, the image $\ph$ evolves keeping itself closer to $c$ in the undamaged region $U$ and diffusing
into the damaged one. However, as we shall see (cf. Rem.~\ref{CHOB} below), $c$ \mod{can only be taken close to $\pm 1$ values, provided that 
$\vert U\vert$ is large enough.}

Originally, reaction terms like \eqref{CHO} and \eqref{Bert} have been proposed in the local case only.
Here we also consider and analyze their nonlocal counterpart.
In addition, in Section \ref{SEC:TUMOR}, we present an application to (solid) tumor growth models, where the presence of the source term $g$ becomes meaningful.
This term, in our abstract models, is motivated by the fact that very often the corresponding systems are coupled with other equations (for instance,
reaction-diffusion processes).
Thus, including $g$ is the first step towards possible generalizations of these results to more refined models.

Here, for a certain class of $S$, we prove the existence of a weak solution on a given time interval and, \mod{in the case \eqref{CHO}}, its uniqueness (see Sections \ref{namr} and \ref{existproof}).
However, note that,
the spatial average of $\rev{\ph}$ instantaneously belongs to $(-1,1)$. Hence, all the known results apply. Further issues and possible generalizations to other
models will be discussed in Section \ref{final}.

\section{Notation, assumptions and main results}
\label{namr}

\subsection{Notation}
To begin with, let us introduce some useful notation.
For a given (real) Banach space $X$, we employ $\norma{\cdot}_X$, $X^*$, and $\< \cdot , \cdot >_X$ to indicate, in the order,
the norm of $X$, its topological dual, and the related duality pairing between $X^*$ and $X$.
Lebesgue and Sobolev spaces defined on $\Omega$, for every $1 \leq p \leq \infty$ and $k \geq 0$,
are denoted by $L^p(\Omega)$ and $W^{k,p}(\Omega)$, with associated norms
$\norma{\cdot}_{L^p(\Omega)}=\norma{\cdot}_{p}$ and $\norma{\cdot}_{W^{k,p}(\Omega)}$, respectively.
When $p = 2$, these become Hilbert spaces and we use $\norma{\cdot}=\norma{\cdot}_2$ for the norm of
$\Lx2$ along with $(\cdot,\cdot)$ for the standard inner product, and account for the standard notation $H^k(\Omega):= W^{k,2}(\Omega)$.
Moreover, we set
\begin{align*}
  & H := \Lx2,
  \qquad
  V := \Hx1,
  \qquad
  W:=\{ v \in \Hx2 : \dn v = 0 \,\, \text{on $\partial \Omega$} \}.
\end{align*}

For any given $v\in V^*$, we define its generalized mean value as
$v_{\Omega}:=\frac1{|\Omega|}\<{v},{1}>_V$, and we recall
the Poincar\'e--Wirtinger inequality:
\begin{align}
  \norma{v - v_{\Omega}}  \leq c_{\Omega} \norma{\nabla v}\rev{,}
  \quad  v\in V.
  \label{poincare}
\end{align}
Here $c_\Omega>0$ depends on $\Omega$ and $d$ only\rev{.}
Besides,
we denote by $V_0$, $H_0$, $V_0^*$ the closed subspaces
of $V$, $H$, and $V^*$, respectively,
consisting of functions (or functionals) with zero spatial
average.
Then, we define the linear bounded operator $\RR : V \to \Vp$ as
\begin{align*}
	\<{\RR  u},{v}>_V &:= \int_\Omega \nabla u \cdot \nabla v  , \quad u,v \in V .
\end{align*}
It is well-known that the restrictions $\RR \vert_{V_0}$ yields an isomorphism from $V_0$ to $V_0^*$
with well defined  inverse
$\RR^{-1}=:\NN : V_0^* \to V_0$.
Furthermore, it holds that
\begin{alignat*}{2}
	\<{\RR  u},{\NN v^*}>_V &= \<{v^*},{u}>_V,
	\qquad && u \in V, \; v^* \in V_0^*,
	\\
	\<{v^*},{\NN w^*}>_V &= \iO  \nabla (\NN v^*) \cdot \nabla (\NN w^*),
	\qquad  && v^*,w^* \in V_0^*.
\end{alignat*}
Also, we have that
\begin{alignat*}{2}
	\norma{v^*}_* &:= \norma{\nabla (\NN v^*)} =
	({\nabla (\NN v^*)},{\nabla (\NN v^*)}) ^{1/2}=\<v^*, \NN v^*>_V^{1/2},
	\qquad &&v^*\in V_0^*,
\end{alignat*}
is a norm in $V_0^*$ with associated inner product denoted by
$(\cdot, \cdot )_{*}:= (\nabla (\NN \cdot),\nabla (\NN \cdot))$.
Similarly, setting
\begin{alignat*}{2}
	\norma{v^*}_{-1} &:= \norma{v^*-(v^*)_\Omega}_* +|(v^*)_\Omega|,
	\quad&&v^*\in V^*,
\end{alignat*}
we get an equivalent norm in $V^*$.
Finally, we recall the identity
\begin{alignat*}{2}
\<{\dt v^*(t)},{\NN v^*(t)}>_V &= \frac 12 \frac d {dt} \norma{v^*(t)}^2_{*}
\quad \text{ for a.e.~}t \in (0,T), \quad&v^* \in \H1 {V^*_0},
\end{alignat*}
and the inequality
\begin{align}\label{oss:mean}
	| v_{\Omega}|\leq c \norma{v}_{\Vp},
	\quad v\in V^*.
\end{align}

Identifying $H$ with its dual $H^*$ by employing the scalar product of $H$, we obtain the
chain of continuous and dense embeddings $V\emb H \emb V^*$.
This implies the following interpolation inequality:
\begin{align}
\label{inter:cpt}
	\norma{v} &= \<{v},{v}>_V ^{1/2}= \<{\RR  v},{\NN v}>_V^{1/2} \leq \norma{v}_{V}^{1/2} \norma{\NN  v}_{V}^{1/2} \leq \norma{v}_{V}^{1/2} \norma{v}_*^{1/2},
	\quad v \in \mod{V_0}.
\end{align}

Moreover, we will also use the following inequality following from Ehrling's lemma and the compact inclusion $V\emb H$:
\begin{align}{}	
	\label{comp:ineq}
	\forall\,\delta>0,\quad\exists\,\cd>0:\quad
	&\norma{v} \leq \d  \norma{\nabla v} + \cd \norma{ v}_{-1}, \quad v \in V.
\end{align}

\subsection{Assumptions}
The following structural assumptions are in order:
\begin{enumerate}[label={\bf A\arabic{*}}, ref={\bf A\arabic{*}}]
\item \label{ass:pot}
The double-well shaped potential $F$ enjoys the splitting $F=\hat \beta + \hat \pi$,
where
\begin{align}
	\label{F:abs:1}
	& \hat \beta \in C^0([-1,1]) \cap C^2(-1,1)
	\text{ is strictly convex, nonnegative, and } \widehat\beta(0)=0,
	\\
	\label{F:abs:2}
	& \hat \pi \in C^1({\erre}),
	\\
	\label{F:abs:3}
	& \pi := \hat \pi':\erre \to \erre \text{ is
	$L_0$-Lipschitz continuous for some $L_0>0$.}
\end{align}
On account of conditions \eqref{F:abs:1}--\eqref{F:abs:3}, using convex analysis we can extend $\hat\beta$ to $+\infty$ outside $[-1,1]$.
Thus the subdifferential $\partial \hat \b =: \beta$ yields a maximal monotone
graph in $\erre \times \erre$ with corresponding domain
$D( \b):=\{ r \in \erre:  \b (r) \neq\emptyset\}$.
In addition, since $\hat \b \in C^1(-1,1)$,
it readily follows that $\b$ is single-valued
with $\beta(0)=0$ and $F'= \beta + \pi$ in $(-1,1)$.
Finally, we suppose that:
\begin{align}
	\label{F:abs:4}
	& \lim_{r \searrow (-1)^{+}} \b(r) = -\infty,
	\quad
	\lim_{r \nearrow 1^-}\b(r) = +\infty,
	\\
	\label{F:abs:5}
	&
	\exists\, \rho \in (0,1):
	\quad 	\b \in L^2(-1,-1+\rho).
\end{align}
\item \label{ass:g}
We assume that $g$ can be written as
\luca{$g=g_0+\tilde g$, where
$g_0\in V$, $\tilde g := g_1 + g_2$,
$g_1\in L^\infty(0,T; H)\cap L^2(0,T; V)$,
and $g_2\in H^1(0,T; H)$.}
Moreover, we require in the local case that
\begin{align}
  \label{ass:g_loc}
	 \forall\,\alpha\in(0,1),\quad\exists\,C_{g}>0:\quad
	\norma{\tilde g}_{L^\infty(0,t; H)}\leq C_gt^{\alpha} \quad\forall\,t\in[0,T],
\end{align}
whereas in the nonlocal case:
\begin{align}
  \label{ass:g_nloc}
  \begin{cases}
	g_0\in W,\\
	\tilde g\in L^\infty(0,T; V),\\
	\forall\,\alpha \in (0,1), \; \exists\,C_g>0:\;
	\norma{\nabla \tilde g}_{L^\infty(0,t; H)}\leq C_gt^\alpha \quad\forall\,t\in[0,T].
  \end{cases}
\end{align}
\item \label{ass:S}
Let the source $S$ be defined as
\begin{align}
	\label{source:splitting}
	S(x,\ph,g):=  - m(x)\ph + h(x,\ph,g),
\end{align}
where
\begin{align}
	\label{source:splitting1}
	&m\in L^\infty(\Omega), \qquad
	m\geq0 \quad\text{a.e.~in } \Omega, \qquad
	 m_\Omega>0,\\
	\label{source:splitting2}
	&h :\Omega\times [-1,1] \times \erre \to \erre \quad\text{is a Carath\'{e}odory function such that} \nonumber\\
	&h(\cdot,u,v)\in L^\infty(\Omega), \quad \forall\,(u,v)\in [-1,1]\times \erre,\\
    &h(x,\cdot,\cdot) \in C^0([-1,1]\times\erre),\quad\text{for a.a.~}x\in \Omega,
\end{align}
and the following compatibility condition between $h$ and $m$ holds:
\begin{align}
	\label{source:splitting3}
	&\exists\,\overline h>0:\;
	|h(x,u,v)| \leq \overline h < m_\Omega,
	\quad\text{for a.a.~}x\in\Omega\rev{,}\,\,\forall\,(u,v)\in[-1,1]\times\erre.
\end{align}
We also set for convenience $\overline m:=\|m\|_{L^\infty(\Omega)}$
and note that $\overline m>0$ by \eqref{source:splitting1}.
\begin{remark}
Assumption \eqref{source:splitting3}
means that the magnitude of $h$ must be controlled by the space-average of $m$.
Note that this allows in particular the degenerate choice
\begin{equation}
\label{char}
  m(x)=
  \begin{cases}
  \overline m \quad&\text{if } x\in U,\\
  0\quad&\text{if } x\in \Omega\setminus U,
  \end{cases}
\end{equation}
where $\overline m>0$ is a fixed constant and
$U\subset \Omega$ is a measurable subset such that  $0<\vert U\vert<\vert \Omega\vert$.
Thus we have $\overline h:=\text{\rm ess sup}_{x,u,v}|h(x,u,v)|< m_\Omega= \frac{|U|}{|\Omega|}\,\overline m$.
\end{remark}

\begin{remark}
\label{CHOB}
Let us set $g=0$ in equation \eqref{sys:2} (\eqref{sys:non:2} resp.).
Then, we observe that the Cahn--Hilliard--Oono equation (see \eqref{CHO}) can be recovered by taking $h$ and $m$ constants, that is,
$$
h(x,u,v) = mc,
\quad x \in \Omega,
$$
where $c\in (-1,1)$ is given. Another meaningful example is provided by the equation proposed for the inpainting (see \eqref{Bert}).
In this case, we have
$$
h(x,u,v) = m(x)c(x)
$$
where $m$ is given by \eqref{char}, $c\in L^\infty(\Omega)$ is such that $c(x)\in[-1,1]$ for almost every $x\in \Omega$, and $\|c\|_{L^\infty(\Omega)}<\frac{|U|}{|\Omega|}$.
Observe that the latter condition is a bit restrictive since $c$ cannot take $\pm 1$ values. \mod{However, we can say that the smaller the damaged region is, the closer to $\pm 1$ $c$ can be taken.
On the other hand, if we consider, for instance, $F$ given by \eqref{Flog}, then
$\varphi$ cannot take the values $\pm1$ either. However, in this case, 
we can take $\vartheta$ small enough in order to have minima as close as possible to $\pm 1$.} \anlast{For further details concerning this issue, see \cite[Subsec.~2.1]{Lam}  and the references provided therein.}
\end{remark}

\mau{
\begin{remark}
\label{KS}
We observe that the Cahn--Hilliard equation with a proliferation term of the form (see \cite{KS}) $S(\ph):=  \alpha(1-\ph^2)$, $\alpha >0$, is not covered by our assumptions.
\end{remark}
}
	
\item  \label{ass:kernel}
The spatial convolution kernel $K \in W^{1,1}_{\rm loc}(\erre^d)$ is such that
\begin{align*}
		K(x)=K(-x),
		\quad
		\text{for a.a. $x\in\erre^d$.}
\end{align*}
  Furthermore, we postulate that
  \[
  a_*:=\inf_{x\in\Omega}\int_\Omega K(x-y)\, { \dd y }=\inf_{x\in\Omega}a(x)\geq0,
  \]
  \[
  a^*:= \sup_{x\in\Omega} \int_\Omega |K(x-y)|\, \dd y< +\infty, \qquad
  b^*:= \sup_{x\in\Omega} \int_\Omega |\nabla K(x-y)|\, \dd y< +\infty,
  \]
  and set $c_a:=\max\{a^*-a_*,1\}>0$.
  Finally, we suppose that
  \[
  \exists\, C_0 >0 : \quad a_* + F''(r)\geq C_0, \qquad\forall\,r\in (-1,1).
  \]
\end{enumerate}

In order to handle the fact that the initial datum
is a pure phase, we will strongly exploit the following qualitative result, \mau{which is one of the main technical novelties of this work.}
\begin{proposition}
\label{prop:pot}
Assume \ref{ass:pot}. Then, the following properties hold.
\begin{enumerate}[start=1,label={{(P\arabic*})}]
  \item \label{PROP:MZ:std}
  For every compact subset $[r_*,r^*]\subset (-1,1)$, $-1 <r_* \leq r^* <1$,
  there exist constants $c_*, C_*>0$ such that,
  for every $r\in(-1,1)$ and for every $r_0\in [r_*,r^*]$ it holds that
  \begin{equation}\label{MZ:std}
  c_*|\b(r)|
  \leq \b(r)(r-r_0)
  +C_*.
  \end{equation}
  \item \label{PROP:MZ:sharp}
  For every $\delta\in(0,1)$, there exist constants $c_\b, C_\b>0$ such that,
  for every $r\in(-1,1)$ and for every $r_0\in(-1,-1+\delta)$
  it holds that
  \begin{equation}\label{MZ:sharp}
  c_\b (r_0+1) |\b(r)|
  \leq \b(r)(r-r_0)+\left[(r+1)
  +C_\b (r_0+1)\right]
  |\b(-1+(r_0+1)/2)|.
  \end{equation}
  In particular, this implies that,
  for every $\delta\in(0,1)$, there exist $c_\b, C_\b>0$ such that,
  for every measurable $\phi:\Omega\to(-1,1)$ with
  $\phi_\Omega\in(-1,-1+\delta)$, we have
  \begin{equation}\label{MZ:sharp'}
  c_\b (\phi_\Omega+1) \int_\Omega|\b(\phi)|\leq
  \int_\Omega\beta(\phi)(\phi-\phi_\Omega)
  +C_\beta(\phi_\Omega+1)
  |\b(-1+(\phi_\Omega+1)/2)|.
  \end{equation}
\end{enumerate}
\end{proposition}
Note that the first property~\ref{PROP:MZ:std} is a classical result in the literature
of Cahn--Hilliard problems, and dates back to the ideas in
\cite[Sec.~5]{GMS09} and \cite{MZ}: this is typically used
to control the $L^1$-norm in space of $\beta(\varphi)$ whenever
the mass of the initial datum, represented by $r_0$ in the above estimates,
 is confined in the interior
of the physical domain $(-1,1)$.
Nevertheless, in our case the initial datum is a pure phase
so, unfortunately, the first property~\ref{PROP:MZ:std} cannot be employed as the constants $c_*$ and $  C_*$ strongly depend on $r_*$ and $ r^*$.
The main technical novelty that we present is then
the refinement of such estimate, i.e.\rev{,} property
\ref{PROP:MZ:sharp}: here, the qualitative behavior of
the constants appearing in \eqref{MZ:std} is
explicitly tracked when $r_0$ approaches the pure phase $-1$.
Note that the constants $c_\beta, C_\beta$ appearing in \eqref{MZ:sharp}
depend only on the threshold value $\delta$, but
are {\em independent of $r$ and $r_0$.}
The role of the parameter $\delta$ is only technical
as it represents the closeness of $r_0$ to the pure phase:
for example, one can think to fix $\delta=\frac13$.
For the corresponding proof, we refer the reader to the Appendix.

We point out that the Flory--Huggins potential \eqref{Flog}
satisfies \ref{ass:pot}. Thus, it also enjoys
the two properties \ref{PROP:MZ:std}--\ref{PROP:MZ:sharp}.
Let us note that
Proposition~\ref{prop:pot} also
holds in the setting of a three-phases Flory--Huggins potential, for
which we refer to \cite[Appendix A]{GSS}.

\subsection{Main results}
The existence results for systems \eqref{sys:1}--\eqref{sys:4} and \eqref{sys:non:1}--\eqref{sys:non:4} read as follow.

\begin{theorem}[local case]\label{THM:EX:LOCAL}
Suppose that \ref{ass:pot}--\ref{ass:S} hold and let the initial datum be the pure phase
\begin{align*}
	\ph_0\equiv-1.
\end{align*}
Then, problem \eqref{sys:1}--\eqref{sys:4}
admits a weak solution $(\ph,\mu)$. Namely, it holds that
\begin{align*}
	& \ph \in \H1 \Vp \cap \L\infty V \cap \L2 W,
	\\
	& \ph \in \Ls4 W \cap \Ls2 {W^{2,\tre}(\Omega)} \quad \forall\, \sigma \in (0,T),
	\\
	& \ph \in \L {2p} W \cap \L {p} {W^{2,\tre}(\Omega)} \quad
	\forall\, p\in (1,\ov p),
	\\
	& \ph \in L^\infty(Q), \quad |\ph(x,t)| <1 \quad\text{ for a.e.~} (x,t) \in Q,
	\\
	& \mu \in \L p  V \cap \Ls2 V \quad
	\forall\, p\in (1,\ov p), \quad  \forall\,\sigma \in (0,T),
    \\
    &\nabla \mu \in \L2 H,
\end{align*}
where
\begin{align}
	\label{def:bp}
\tre\in\begin{cases}
[1,6] \quad&\text{if } d=3,\\
[1,+\infty) \quad&\text{if } d=2,
\end{cases}
\qquad\qquad
\ov p:=
\begin{cases}
  \frac43 \quad&\text{if } g_0\in V\setminus W,\\
  2\quad&\text{if } g_0\in W,
\end{cases}
\end{align}
such that
\begin{align*}
		\mu=  - \Delta \ph + F'(\ph) +g, \quad & \text{$a.e.$ in $Q$},
\end{align*}
and
the variational equality
\begin{align*}
	& \<\dt \ph ,v>_{V}
	+ \iO \nabla \mu \cdot \nabla v = \iO S(\cdot,\ph,g) v
\end{align*}
is satisfied for every test function $ v\in V$, and almost everywhere in $(0,T)$.
Moreover, it holds that
\begin{align*}
	\ph(0)=-1 \quad \text{a.e.~in $\Omega.$}
\end{align*}
\end{theorem}
\begin{theorem}[nonlocal case]\label{THM:EX:NONLOCAL}
Suppose that \ref{ass:pot}--\ref{ass:kernel} hold and let the initial datum be the pure phase
\begin{align*}
	\ph_0\equiv-1.
\end{align*}
Then, problem \eqref{sys:non:1}--\eqref{sys:non:4}
admits a weak solution $(\ph,\mu)$. Namely, it holds that
\begin{align*}
	& \ph \in \H1 \Vp \cap \L\infty H \cap \L2 V ,
	\\	
	& \ph \in L^\infty(Q), \quad |\ph(x,t)| <1 \quad\text{ for a.e.~} (x,t) \in Q,
	\\
	& \mu \in \L p  V \cap \Ls2 V \quad \forall\,p\in (1,2),
	\quad  \forall\, \sigma \in (0,T),
    \\
    &\nabla \mu \in \L2 H,
\end{align*}
such that
\begin{align*}
	\mu=  a\ph - K* \ph + F'(\ph) +g, \quad & \text{$a.e.$ in $Q$},
\end{align*}
and
the variational equality
\begin{align*}
	& \<\dt \ph ,v>_{V}
	+ \iO \nabla \mu \cdot \nabla v = \iO S(\cdot,\ph,g) v
\end{align*}
is satisfied for every test function $ v\in V$, and almost everywhere in $(0,T)$.
Moreover, it holds that
\begin{align*}
	\ph(0)=-1 \quad \text{a.e.~in $\Omega.$}
\end{align*}
\end{theorem}
\begin{remark}
It is worth mentioning that the standard assumptions
$\varphi_0\in V$, $F(\varphi_0)\in L^1(\Omega)$,
and $(\ph_0)_\Omega \in (-1,1)$ entail that all regularities in the above theorems also hold for
$\sigma=0$ as the usual approach can be applied.
The same goes if the initial condition $\ph(0)\equiv-1$ is approximated by
$\ph(0)=\ph_{0,\delta} :=-1 + \delta$ for some $\delta \in (0,1)$.
\end{remark}

\mod{Finally, assuming that $h $ is constant, we \anlast{establish} a continuous dependence estimate on the initial datum, which \anlast{may assume the value}  $-1$. 
Extending this result to more general reaction terms \anlast{remains} an open issue.
Under this additional assumption, we \anlast{obtain the following result.}}
\begin{theorem}[local case]\label{THM:UNI:LOCAL}
Suppose that \ref{ass:pot}--\ref{ass:S} hold and \anlastr{$h$ is independent of the phase variable $\ph$}.
Let $\varphi_{0j}\in V$ be such that $(\ph_{0j})_\Omega\in [-1,1)$, $j=1,2$
and denote by $\ph_j$ a weak solution originating from $\ph_{0j}$. Then, there exists
$C>0$, depending on $T$, $\Omega$, $F$, $m$, and $h$, such that
\begin{align}
	\non
	&
	\norma{(\ph_2 - \ph_1)(t)}^2_{-1}
	+ |(\ph_2-\ph_1)_\Omega(t)|
+ \int_0^t \Vert (\ph_2 - \ph_1)(s)\Vert^2_V \,\ds
	\\ & \quad 	   \label{loccd}
	\leq
	 C\Vert \ph_{02}-\ph_{01}\Vert^2_{-1}
+C|(\ph_{02})_\Omega-(\ph_{01})_\Omega|, \quad\forall\,t\in [0,T].
\end{align}
\end{theorem}
\begin{theorem}[nonlocal case]\label{THM:UNI:NONLOCAL}
Suppose that \ref{ass:pot}--\ref{ass:kernel} hold and \anlastr{$h$ is independent of the phase variable $\ph$}.
Let $\varphi_{0j}\in H$ be such that $(\ph_{0j})_\Omega\in [-1,1)$, $j=1,2$ and denote by $\ph_j$ a weak solution originating from $\ph_{0j}$. Then, there exists
$C>0$, depending on $T$, $\Omega$, $K$, $F$, $m$, and $h$, such that
\begin{align}
	\non
	& \norma{(\ph_2 - \ph_1)(t)}^2_{-1}
+|(\ph_2-\ph_1)_\Omega(t)|	
	+ \int_0^t \Vert (\ph_2 - \ph_1)(s)\Vert^2 \,\ds
	\\ & \quad \label{nonloccd}
	\leq
	 C\Vert \ph_{02}-\ph_{01}\Vert^2_{-1}
+C|(\ph_{02})_\Omega-(\ph_{01})_\Omega|	, \quad\forall\,t\in [0,T].
\end{align}
\end{theorem}
Thus, we can conclude that our problems are well posed, provided that \anlastr{the function $h$ does not depend on the phase variable $\ph$. This holds, in particular, for the Cahn–Hilliard–Oono equation (see \eqref{CHO}), where $h$ is a constant, as well as for the inpainting equation (see \eqref{Bert}).}

\section{Existence of weak solutions}
\label{existproof}
The strategy of the proofs of Theorem~\ref{THM:EX:LOCAL} and Theorem~\ref{THM:EX:NONLOCAL} is outlined here below.
The local and nonlocal cases slightly differ in the type of estimates and regularities we can obtain.
Indeed, the fourth--order nature of the local Cahn--Hilliard equation allows higher regularity with respect to the nonlocal one.

\begin{enumerate}[label={\bf (\Roman{*})}]
\item {\bf Approximation.}
\luca{The first step consists in considering suitable approximate problems. Both in
the local and the nonlocal case, the approximation only consists in shifting to the right
the pure phase initial datum $-1$, so that we can recover the classical assumption
that the initial mass is strictly confined in the physical domain $(-1,1)$. More precisely,
for every approximation parameter $\lambda\in(0,1)$ one consider\rev{s} the family
of problems starting from the initial datum $\varphi_{0,\lambda}:=-1+\lambda$ instead,
and with respect to the same singular potential $F$.}

\item {\bf Preliminary estimate on the phase variable.} By exploiting the Cahn--Hilliard structure, a suitable preliminary, uniform with respect to $\lambda$, estimate on the phase variable $\ph_\lambda$ can be derived without involving the chemical potential $\mu_\lambda$. \luca{This is the first information to evaluate the rate at which $\varphi_\lambda$ reaches its initial
value for small times.}

\item {\bf \luca{Qualitative analysis of the mean value.}} Before proceeding with further uniform estimates, a control of the spatial mean of $\ph_\lambda$ is necessary.
\mau{In particular, in order to exploit \ref{PROP:MZ:std} and \ref{PROP:MZ:sharp},
one needs to prove that $(\varphi_\lambda)_\Omega$
remains detached from $-1$ for positive times, uniformly in $\lambda$.
This is not obvious as the initial datum only satisfies
$(\ph_\lambda(0))_\Omega=-1+\lambda$. The qualitative behavior of
$(\ph_\lambda)_\Omega$ is obtained via an analysis of the ODE in \eqref{ODEmean}. This allows to show that
the mean values instantly detaches from $-1$ at a linear rate, uniformly in $\lambda$.
}

\item {\bf Energy estimate.} By using the above information, the standard Cahn--Hilliard weak energy estimate can be now performed on the $\lambda$-approximation level. Namely, we can test the (approximated version of the) first equation \eqref{sys:1} (\eqref{sys:non:1} resp.) by $\mu_\lambda$ and the second \eqref{sys:2} (\eqref{sys:non:2} resp.) by $\dt \ph_\lambda$.
As is well known for Cahn--Hilliard-type systems, this procedure produces just a $L^2$-control of $\nabla \mu_\lambda$ and no information is available for $\mu_\lambda$ itself due the homogeneous Neumann boundary conditions. This issue is usually overcome by the Poincar\'e--Wirtinger inequality provided to show that $(\mu_\lambda)_\Omega$ is bounded in $L^2(0,T)$. Note that, in the present case, we also need to handle the inner product $(\rev{S(\cdot, \phl,g)},\mu_\lambda)$.
\luca{To this end,} comparison in the second equation \eqref{sys:2}  (\eqref{sys:non:2} resp.) reveals that this, roughly speaking, equals to controlling the singular term $\rev{\b}(\ph_\lambda)$ in $\L2 {L^1(\Omega)}$.
However, the singular nature of $\rev{\b}$ and the pathological initial condition $\phl(0)=-1$ prevent us to get such a bound through the usual approach based on \ref{PROP:MZ:std}. Nevertheless,
\luca{hinging on \ref{PROP:MZ:sharp},
we are able to suitably combine the qualitative estimates on the behavior of $(\varphi_\lambda)_\Omega$ with the blow-up of the nonlinearity
and to close the energy estimate.
This yields directly a first $L^2$-control of $\nabla \mu_\lambda$.
In order to gain a full control on the $H^1$-norm of $\mu_\lambda$
we are only left to show specific integrability properties of $(\mu_\lambda)_\Omega$.
These follow again by a sharp qualitative analysis of the behavior of
$(\varphi_\lambda)_\Omega$, making use of \ref{PROP:MZ:std}--\ref{PROP:MZ:sharp}.
The only drawback in having a pure phase initial datum results in a
slight loss of integrability of $\mu_\lambda$ close to the initial time.
Notice that such a loss of information is almost negligible
and seems reasonable to us in view of the initial pure phase.
Indeed, by comparing the classical scenario of
Cahn--Hilliard equations with ours,
assuming $g$ to fulfill either \eqref{ass:g_loc} or \eqref{ass:g_nloc},
it turns out that the correct spaces where one has to bound the chemical potentials
$\mu_\lambda$ are
\begin{align*}
  \underbrace{L^2(0,T; H^1(\Omega))}_{\text{in the classical setting}}
  \qquad\text{vs}\qquad
	\underbrace{\bigcap_{\s \in (0,T), \, {p \in (1,\ov p)}}L^2(\s,T; {H^1(\Omega)}) \cap L^ p(0,\sigma; {H^1(\Omega)}) }_{\text{in our setting}}.
\end{align*}
As noted before, the loss of integrability on the chemical potential
caused by the pure phase initial datum is extremely light, and absolutely
reasonable if taking into account the singular nature the problem.}

\item {\bf Passing to the limit.} \mau{The final step consists in realizing that all the previous estimates are independent of $\lambda$ so that weak and weak star compactness arguments can be exploited to rigorously justify the passage to the limit along with a suitable sequence  $\lambda_k\searrow 0$.
This procedure provides a \rev{limit} pair $(\ph,\mu)$ which fulfills suitable regularities and, in a suitable sense, yields a solution to the original problem \eqref{sys:1}--\eqref{sys:4} (\eqref{sys:non:1}--\eqref{sys:non:4} resp.).
We highlight that, as expected from the above lines, the time integrability of $\mu$ becomes the usual one, as soon as we detach from $t=0$, i.e., on every time interval of the form $(\sigma,T)$ for all $\s \in (0,T)$. 
}
\end{enumerate}

Since the above plan works for both the proofs of Theorem~\ref{THM:EX:LOCAL} and Theorem~\ref{THM:EX:NONLOCAL}, we are going to present them in a unified fashion, just pointing out the main differences between the local and nonlocal cases.
The proof will be divided into the aforementioned steps.

\subsection{Step (I)}

In what follows, we are going to indicate with $C$ a generic positive constant, independent of $\lambda$,
whose value may change from line to line.
We now fix $\l \in (0,1)$ \rev{and} introduce the $\lambda$-approximation of problems \eqref{sys:1}--\eqref{sys:4} and \eqref{sys:non:1}--\eqref{sys:non:4}, respectively.

\noindent
{\bf Local case.}
We consider the approximating problem
\begin{alignat}{2}
	\label{SYS:1:app}
	& \dt \ph_\lambda - \Delta\mu_\lambda
	= S(x,\ph_\lambda,g)
		&&\qquad \text{in $Q$},
		\\	\label{SYS:2:app}
	& \mu_\lambda =
	-\Delta \ph_\lambda
	+ F'(\ph_\lambda)
	 + g
		&&\qquad \text{in $Q$},
		\\
	\label{SYS:3:app}
	& \dn\ph_\lambda = \dn\mu_\lambda   =0
		&&\qquad \text{on $\Sigma$},\\
	\label{SYS:4:app}
	& \ph_\lambda(0)
	= {\ph}_{0,\lambda}:=-1 + \l
	&&\qquad \text{in $\Omega$}.
\end{alignat}
Thanks to \sub{previous} results \luca{(see, e.g.,~\cite{Mir})}, the above problem has \mau{a weak solution} $(\ph_\lambda, \m_\lambda)$  such that
\begin{align*}
	\ph_\lambda  \in H^1(0,T; V^*) \cap \L\infty V \cap \L2 {H^2(\Omega)},
	\quad
	\mu_\lambda  \in \L2 V.
\end{align*}
Moreover, due to the singularity of the potential $F$, we have $\phl \in L^\infty(Q)$ with
\begin{align*}
	-1 \leq \phl(x,t)\leq 1
	\quad \text{ for a.e.} \,(x,t) \in Q
\end{align*}
whence also the corresponding boundedness of $S$ using the structural assumptions in \ref{ass:S}.

\noindent
{\bf Nonlocal case.}
In this case, we consider the following approximating problem
\begin{alignat}{2}
	\label{SYS:1:app:nonloc}
	& \dt \ph_\lambda - \Delta\mu_\lambda
	= S(x,\ph_\lambda,g)
		&&\qquad \text{in $Q$},
		\\	\label{SYS:2:app:nonloc}
	& \mu_\lambda =
	a \ph_\lambda
	-  K*\ph_\lambda
	+ F'(\ph_\lambda)
		+ g
		&&\qquad \text{in $Q$},
		\\
	\label{SYS:3:app:nonloc}
	& \dn\mu_\lambda   =0
		&&\qquad \text{on $\Sigma$},\\
	\label{SYS:4:app:nonloc}
	& \ph_\lambda(0)
	:= {\ph}_{0,\lambda}=-1 + \l
	&&\qquad \text{in $\Omega$}.
\end{alignat}
\mau{We can adapt the argument
in \cite[Sec.~3]{FG} for the existence (see also references therein).
Summing up, there exists \mau{a weak solution}
$(\ph_\lambda, \m_\lambda)$
to the above problem, such that}
\begin{align*}
	\ph_\lambda  \in H^1(0,T; V^*)\cap \L\infty H \cap \L2 V,
	\quad
	\mu_\lambda  \in \L2 V.
\end{align*}
Moreover, arguing as above, it holds that $\phl \in L^\infty(Q)$ with
\begin{align*}
	-1 \leq \phl(x,t)\leq 1
	\quad \text{ for a.e.}\,(x,t) \in Q,
\end{align*}
and the corresponding boundedness of $S$.

\subsection{Step (II)}
Let us now establish some uniform a priori estimates for the phase field variable $\ph_\lambda$.

\noindent
{\bf Local case.}
Recalling \eqref{ass:g_loc}, \ref{ass:g}, and that $\ph_{0,\l}=-1+\l$,
we test equation \eqref{SYS:1:app} by $\ph_\lambda-\ph_{0,\l}=\varphi_\lambda+1-\l$,
equation \eqref{SYS:2:app} by
$-\Delta(\ph_\lambda-\ph_{0,\l})=-\Delta\varphi_\lambda$,
and add the resulting identities together to obtain that
\begin{align*}
	&\frac 12 \frac d{dt} \norma{\ph_\lambda-\ph_{0,\l}}^2
	+ \norma{\Delta \ph_\lambda}^2
	- \iO \b(\ph_\lambda) \Delta \ph_\lambda\\
	&\quad=
	\iO S(\cdot,\ph_\lambda,g) (\ph_\lambda-\ph_{0,\l})
	+ \iO \pi(\ph_\lambda)\Delta \ph_\lambda
	+  \iO g_0\Delta \ph_\lambda
	+  \iO \tilde g\Delta \ph_\lambda.
\end{align*}
As for the third term on the \lhs, integrating by parts and using the convexity of $\hat \b$, we get
\begin{align*}
	- \iO \b(\ph_\lambda) \Delta \ph_\lambda
	= \iO \b'(\ph_\lambda) |\nabla \ph_\lambda|^2 \geq 0.
\end{align*}
Moreover, thanks to the boundedness of $S$, the first term on the right-hand side
can be bounded as
\begin{align*}
	& \iO S(\cdot,\ph_\lambda,g)  (\ph_\lambda-\ph_{0,\l}) \leq
	C \norma{ \ph_\lambda-\ph_{0,\l}},
\end{align*}
while using \ref{ass:pot} and integration by parts, we get
\[
  \iO \pi(\ph_\lambda)\Delta \ph_\lambda
	\leq L_0\norma{\nabla\ph_\lambda}^2
	=L_0\norma{\nabla(\ph_\lambda-\ph_{0,\l})}^2
	\leq\frac14\norma{\Delta \ph_\lambda}^2
	+C\norma{ \ph_\lambda-\ph_{0,\l}}^2.
\]
Furthermore, recalling \ref{ass:g} and integrating by parts yield
\begin{equation}
\label{estg}
  \iO g_0\Delta \ph_\lambda
	+  \iO \tilde g\Delta \ph_\lambda
 \leq C\left(\norma{\nabla\varphi_\lambda}
 +\luca{t^{\alpha}\norma{\Delta\varphi_\lambda}}\right)
 \leq \frac14\norma{\Delta\varphi_\lambda}^2
 +C\left(\luca{t^{2\alpha}} + \norma{\nabla\varphi_\lambda}\right).
\end{equation}
Upon rearranging the terms, we obtain that
\begin{align*}
	&\frac 12 \sup_{r\in[0,t]}\norma{\ph_\lambda-\ph_{0,\l}}^2
	+ \frac 12\int_0^t\norma{\Delta \ph_\lambda(s)}^2\,\ds\\
	&\quad\leq \luca{Ct^{2\alpha+1}} + C\norma{\nabla\varphi_\lambda}_{L^1(0,t; H)} +
	 C \int_0^t(
	 \norma{ \ph_\lambda(s)-\ph_{0,\l}}
	 +\norma{ \ph_\lambda(s)-\ph_{0,\l}}^2)\,\ds\\
	 &\quad\leq
	 \luca{Ct^{2\alpha+1}} + Ct^{\frac34}\norma{\nabla\varphi_\lambda}_{L^4(0,t; H)}
	 +C \int_0^t\big(
	 \norma{ \ph_\lambda(s)-\ph_{0,\l}}
	 +\norma{ \ph_\lambda(s)-\ph_{0,\l}}^2\big)\,\ds.
\end{align*}
By interpolation and the Young inequality, it readily follows that
\begin{align*}
Ct^{\frac34}\norma{\nabla\varphi_\lambda}_{L^4(0,t; H)}
&=Ct^{\frac34}\norma{\nabla(\varphi_\lambda-\ph_{0,\l})}_{L^4(0,t; H)}\\
&\leq Ct^{\frac34}\norma{\Delta\varphi_\lambda}_{L^2(0,t; H)}^{\frac 12}
\norma{\varphi_\lambda-\ph_{0,\l}}_{L^\infty(0,t; H)}^{\frac 12}\\
&\leq Ct^{\frac32} + \frac1{2}\norma{\Delta\varphi_\lambda}_{L^2(0,t; H)}
\norma{\varphi_\lambda-\ph_{0,\l}}_{L^\infty(0,t; H)}\\
&\leq Ct^{\frac32}
 +\frac1{4}\norma{\Delta\varphi_\lambda}^2_{L^2(0,t; H)}
 +\frac14\norma{\varphi_\lambda-\ph_{0,\l}}_{L^\infty(0,t; H)}^2,
\end{align*}
from which, collecting all the terms, we deduce that
\begin{align}
\nonumber
	&\frac 14 \sup_{r\in[0,t]}\norma{\ph_\lambda(r)-\ph_{0,\l}}^2
	+ \frac 1{4}\int_0^t\norma{\Delta \ph_\lambda(s)}^2\,\ds\\
\label{step2:local}
	 &\quad\leq
	 \luca{C (t^{2\alpha+1} + t^{\frac32})}
	 +C \int_0^t\big(
	 \norma{ \ph_\lambda(s)-\ph_{0,\l}}
	 +\norma{ \ph_\lambda(s)-\ph_{0,\l}}^2\big)\,\ds.
\end{align}
Thus, Gronwall's lemma and elliptic regularity theory yield
\begin{align}\label{prelim:reg:local}
	\norma{\ph_\lambda}_{\L\infty H \cap \L2 W}
	\leq
	C.
\end{align}
Using this information,
we can go back to \eqref{step2:local} and bound the \rhs\ and infer that
\begin{align*}
	\norma{\ph_\lambda- \ph_{0,\l}}_{L^\infty (0,t;H) \cap L^2(0,t; W)}
	\leq
	C t^{1/2},
	\quad
	\forall t \in [0,T].
\end{align*}
\luca{Noting that $2\alpha+1\geq\frac32$ for some $\alpha\in(0,1)$,}
iterating the argument on the right-hand side of \eqref{step2:local}
yields
\begin{align}\label{est:alpha:local}
	\forall \alpha \in [0,3/4],
	\quad
	\exists\, C_\alpha>0:
	\quad
	\norma{\ph_\lambda- \ph_{0,\l}}_{L^\infty (0,t;H) \cap L^2(0,t; W)}
	\leq
	C_\alpha t^{\alpha},
	\quad
	\forall\,t \in [0,T],
\end{align}
where the constants $C_\alpha$ are independent of $\lambda$.
Let us note that if, besides \eqref{ass:g_loc}, it also holds that
$g_0\in W$, one can see that \maunew{(cf. \eqref{estg})}
\begin{align}
\label{stima:gzero}
  \iO g_0\Delta \ph_\lambda=\iO g_0\Delta (\ph_\lambda-\ph_{0,\l})
  =\iO \Delta g_0 (\ph_\lambda-\ph_{0,\l})
  \leq C\norma{\ph_\lambda-\ph_{0,\l}}.
\end{align}
Hence, the same argument {above allows us} to infer \eqref{est:alpha:local}, for every $\alpha\in[0,1)$.

\noindent
{\bf Nonlocal case.}
We can argue as above.
Namely, we test \eqref{SYS:1:app:nonloc} by $ \phl-\ph_{0,\l}$, the gradient of equation \eqref{SYS:2:app:nonloc} by $ {- \nabla (\phl - \ph_{0,\l})=}- \nabla \phl$, and add the resulting identities.
This gives
\begin{align*}
	& \frac 12 \frac d{dt} \norma{\ph_\lambda-\ph_{0,\l}}^2
	+ \iO F''(\ph_\lambda) |\nabla  \ph_\lambda|^2
	+ \iO \nabla (a \phl - K*\phl)\cdot \nabla  \ph_\lambda
	\\ & \quad
	=
	\iO S(\cdot,\ph_\lambda,g) (\ph_\lambda-\ph_{0,\l})
	- \iO \nabla g_0 \cdot \nabla \ph_\lambda
	- \iO \nabla \tilde g \cdot \nabla \ph_\lambda.
\end{align*}
We note straightaway that, being $\varphi_{0,\l}=-1+\l$,
by definition of $a$, one has that
\[
\iO \nabla (a \ph_{0,\l} - K*\ph_{0,\l})\cdot \nabla  \ph_\lambda
=\iO \nabla (-a +\l a+a -\l a)\cdot \nabla  \ph_\lambda = 0 .
\]
Taking this information into account,
the second and third terms on the \lhs\ can be bounded
by recalling \ref{ass:kernel} along with Young's inequality as
\begin{align*}
	& \iO F''(\ph_\lambda) |\nabla  \ph_\lambda|^2
	+ \iO \nabla (a \phl - K*\phl)\cdot \nabla  \ph_\lambda
	\\ & \quad
	=\iO F''(\ph_\lambda) |\nabla  \ph_\lambda|^2
	+ \iO \nabla {\big(}a (\phl-\ph_{0,\l}) - K*(\phl-\ph_{0,\l}){\big)}\cdot \nabla  \ph_\lambda
	\\ & \quad
	=
	\iO (F''(\ph_\lambda) + a) |\nabla  \ph_\lambda|^2
	+\iO  {\big(} (\nabla a) (\phl-\ph_{0,\l})  - (\nabla K )*(\phl-\ph_{0,\l}){\big)} \cdot \nabla  \phl
	\\
	& \quad \geq
	C_0  \norma{\nabla \phl}^2
	- 2b^* \norma{\phl-\ph_{0,\l}}\norma{\nabla \phl}
	\\
	& \quad \geq
	\frac  {C_0}2 \norma{\nabla \phl}^2
	-C\norma{\phl-\ph_{0,\l}}^2.
\end{align*}
On the other hand, recalling
\ref{ass:g}, \eqref{ass:g_nloc}, and using the boundedness of $S$,
integration by parts and the {Young} inequality, we obtain
\begin{align*}
	& \iO S(\cdot,\ph_\lambda,g) (\ph_\lambda-\ph_{0,\l})
	- \iO \nabla g_0 \cdot \nabla \ph_\lambda
	- \iO \nabla \tilde g \cdot \nabla \ph_\lambda
	\\ &\quad=
	 \iO S(\cdot,\ph_\lambda,g) (\ph_\lambda-\ph_{0,\l})
	+ \iO \Delta g_0  (\ph_\lambda-\ph_{0,\l})
	- \iO \nabla \tilde g \cdot \nabla \ph_\lambda
	\\ &\quad\leq
	C\norma{\phl-\ph_{0,\l}}+ \luca{Ct^\alpha}\norma{\nabla\phl}
	\\ &\quad\leq
	C\norma{\phl-\ph_{0,\l}} +
	\frac  {C_0}4 \norma{\nabla \phl}^2
	+ \luca{Ct^{\alpha+1}}.
\end{align*}
Thus, upon rearranging terms, we find
\begin{align}
\nonumber
	&\frac 12 \norma{\ph_\lambda(t)-\ph_{0,\l}}^2
	+ \frac {C_0}4\int_0^t\norma{\nabla \ph_\lambda(s)}^2\,\ds\\
	&\quad\leq \luca{Ct^{\alpha+1}}+
	 C \int_0^t\big(\norma{ \ph_\lambda(s)-\ph_{0,\l}}+
	 \norma{ \ph_\lambda(s)-\ph_{0,\l}}^2\big)\,\ds
\label{step2:nonlocal}
\end{align}
and Gronwall's lemma entails
\begin{align}\label{prelim:reg:nonlocal}
	\norma{\ph_\lambda }_{\L\infty H \cap \L2 V}
	\leq
	C.
\end{align}
Then, \luca{noting that $\alpha+1\in(1,2)$ for all $\alpha\in(0,1)$,}
we argue as before by iteration
to infer that
\begin{align}\label{est:alpha:nonlocal}
	\forall \,\alpha \in [0,1),
	\quad
	\exists\, C_\alpha>0:
	\quad
	\norma{\ph_\lambda - \ph_{0,\l}}_{L^\infty (0,t;H) \cap L^2(0,t; V)}
	\leq
	C_\alpha t^{\alpha},
	\quad
	\forall\, t \in [0,T],
\end{align}
where the constants $C_\alpha$ are independent of $\lambda$.

\subsection{Step (III)}
Unlike the previous steps, this one is independent of the choice of the local or nonlocal model as the mass evolution
of $\ph$ is always ruled by \rev{the first equation in \eqref{ODEmean}} with initial condition $\varphi_{0,\l}$.
More precisely, the aim of this step is to prove the following \rev{result.}
\begin{lemma}
\label{lem:mean}
  In the current setting, there exists
  a constant $c_0\in(0,1)$,
  depending only on $m$ and $h$,
  such that, \luca{for all $\lambda\in(0,1)$ and $t\in[0,T]$,}
  \begin{align}\label{claim:mean}
	-1 \leq -1 + c_0 (1- e^{-\overline mt})
	\leq (\phl(t))_\Omega
	\leq \luca{-1 + (2-c_0) (1-e^{-\overline m t}) +\lambda e^{-\overline mt}}.
\end{align}
In particular, there are constants \luca{$c_1,c_2>0$ and $\l_0 \in (0,1)$,
depending only on $m$ and $h$,} such that, for all $\l \in (0,\l_0)$,
  \begin{align}
  \label{claim:mean2}
	-1 +2c_1t
	\leq (\phl(t))_\Omega
	\luca{\leq -1 +\l +c_2(1-e^{-\overline m t}) \leq 1-\lambda_0<1}
	\qquad
	\forall \,t \in [0,T].
\end{align}
\end{lemma}
\begin{proof}[Proof of Lemma~\ref{lem:mean}]
Integrating \eqref{SYS:1:app} (\eqref{SYS:1:app:nonloc} resp.) over $\Omega$ and multiplying by $1 /|\Omega|$ lead to a Cauchy problem for the mean of the phase variable $y_\l:=(\ph_\lambda)_\Omega$
(cf. \eqref{ODEmean})
\begin{align*}
		y_\l' = (S(\cdot,\ph_\lambda,g) )_\Omega,
		\quad
		y_\l(0)=-1+\l,
\end{align*}
Thus, recalling the boundedness of $S$, the associated Cauchy problem for $y_\l$
reads as
\begin{align*}
	\begin{cases}
	y'_\l + (m{(\cdot)}\phl)_\Omega = (h(\cdot,\phl,g))_\Omega \quad \text{in $(0,T)$},
	\\
	y_\l(0)=-1+\l,
	\end{cases}
\end{align*}
which can be equivalently {reformulated} as
\begin{equation}\label{eq_mean_aux}
\begin{cases}
	y'_\l + \overline m y_\l = (h(\cdot,\phl,g))_\Omega +
	\overline m y_\l-({m(\cdot)}\phl)_\Omega \quad \text{in $(0,T)$},
	\\
	y_\l(0)=-1+\l.
	\end{cases}
\end{equation}
Recalling that $\overline m=\|m\|_{L^\infty(\Omega)}>0$ and noting that
\begin{align*}
	\overline m y_\l-({m(\cdot)}\phl)_\Omega
	=
	{\big((\ov m-m(\cdot))\phl\big)_\Omega},
\end{align*}
we have the solution formula
\begin{align}
  \non
  y_\lambda(t)& =(-1+\l)e^{-\overline m t} + \int_0^te^{-\overline m(t-s)}
  {\Big(}(h(\cdot,\phl(s),g(s)))_\Omega +
	((\overline m-m{(\cdot)})\phl(s))_\Omega
	{\Big)}\,\ds.
	\label{est:mean}
\end{align}
Let us focus on the estimate from below.
Recalling that $|h|\leq \overline h$,
we deduce that
\begin{align*}
  y_\lambda(t)+1
  &
  \geq(1  -e^{-\overline m t})
  +\l e^{-\overline m t}
  - \frac{\overline h}{\overline m}(1-e^{-\overline m t})
  +\int_0^te^{-\overline m(t-s)}
  ((\overline m-m{(\cdot)})\phl(s))_\Omega\,\ds.
\end{align*}
Now, since $|\phl|\leq1$ and $0\leq m\leq \overline m$ in $\Omega${,}
we have
\begin{align*}
  ((\overline m-m{(\cdot)})\phl)_\Omega
  &=\frac1{|\Omega|}\int_\Omega(\overline m-m{(\cdot)})\phl
  \geq-\frac1{|\Omega|}\int_\Omega(\overline m -m{(\cdot)})
  =-\overline m + m_\Omega.
\end{align*}
Putting everything together, we infer that
\begin{align*}
y_\lambda(t) + 1  &\geq {\Big(}
1 -\frac{\overline h}{\overline m}
-1
+\frac{m_\Omega}{\overline m}
\Big)
(1-e^{-\overline m t})
=\frac1{\overline m}{\Big(}
m_\Omega-\overline h
{\Big)}
(1-e^{-\overline m t}).
\end{align*}
Since $\overline h<m_\Omega\leq \overline m$,
the estimate from below in \eqref{claim:mean}
follows with the choice
\[
  c_0:=\frac{m_\Omega-\overline h}{\overline m}\in(0,1).
\]
Let us focus now on the estimate from above.
Recalling that $|h|\leq \overline h$, $|\phl|\leq1$, and
$0\leq m\leq\overline m$ in $\Omega$, we infer that
\begin{align*}
  \luca{y_\lambda(t)+1}&\leq \luca{(1 -e^{-\overline m t})}
  +\l e^{-\overline m t}+
  \frac{\overline h}{\overline m}(1-e^{-\overline mt})
  +\Big(\frac{\overline m - m_\Omega}{\overline m}
  \Big)(1-e^{-\overline mt})\\
  &=\luca{\left(2-\frac{m_\Omega-\overline h}{\overline m}\right)(1-e^{-\overline m t})
  +\lambda e^{-\overline mt}}
  \\ &
  = \luca{(2-c_0)(1-e^{-\overline mt})
  +\lambda e^{-\overline mt}},
\end{align*}
and \luca{\eqref{claim:mean} is proved.
The second inequality in \eqref{claim:mean} is then a direct consequence
of \eqref{claim:mean} by taking, for instance,
$\l_0 := \frac{c_0}2$.}
\end{proof}

\begin{remark}
\label{REM:regsigma}
We highlight that Lemma~\ref{lem:mean} ensures that
\begin{align*}
	\luca{(\phl(t))_\Omega \in [-1,1), \quad 	\forall \,t\in [0,T].}
\end{align*}
This will be a crucial property to motivate the interpretation of  the chemical potential $\ml$ for small times in the limit $\l\to0$.
Besides, {as a} straightforward consequence of the above analysis,
{we have that}
\begin{align}\label{mean:dopozero}
	\forall\, \s \in (0,T), \quad \exists\, c_\s \in(0,1) : \quad
	-1+c_\s
	\leq
	(\phl(t))_\Omega
	<
	1-c_\s,
	\quad
	\forall\,t \in [\s,T],
\end{align}
meaning that for every positive time $\sigma${,} the mass of the phase variable $\phl$ is confined in the physical interval $(-1,1)$ {independently of $\l$}.
This is the standard condition that is usually required in the Cahn--Hilliard literature to infer from inequality \eqref{MZ:std} that $\b(\phl) \in L^2(\sigma,T; {\Lx1})$.
Hence, from comparison in \eqref{sys:2} (\eqref{sys:non:2} resp.), it
follows that $(\ml)_\Omega \in L^2(\sigma,T)$ and, via the Poincar\'e--Wirtinger inequality and the above energy estimate, one recovers that $ \ml \in L^2(\sigma,T;V)$.
\end{remark}

\subsection{Step (IV)}
This step is devoted to the derivation of the energy estimates.
The idea is to combine the classical energy estimate \luca{with the enhanced version
\eqref{MZ:sharp} of the inequality \eqref{MZ:std}. This enables to derive insights into the $L^1$-spatial norm of the nonlinear term $\beta(\varphi_\lambda)$ by the use}
of the evolution properties of the mean of $\phl$ and the preliminary estimate.

\noindent
{\bf Local case.}
We test \eqref{SYS:1:app} by $\ml$, \eqref{SYS:2:app} by $- \dt \phl$,
add the resulting equalities, and integrate in time, for any arbitrary $t \in (0,T]$. This gives
\begin{align}
	&\non
	\frac 12 \norma{\nabla \phl(t)}^2  + \iO F(\phl(t))
	+\int_0^t \norma{\nabla \ml(s)}^2\,\ds \\
	&\quad=|\Omega|F(-1+\l)+
	\int_0^t\iO S(\cdot,\phl(s),g(s)) \ml(s)\,\ds
	\luca{- \int_0^t\iO g(s) \dt \phl(s)\,\ds}.
	\label{est:energy:1}
\end{align}
It readily follows from the previous estimates that
\begin{align*}
	\iO F(\phl(t))
	\geq
	\iO \hat \b(\phl(t))
	- C.
\end{align*}
As for the last term on the right-hand side, we point out that the notation is only formal.
More precisely, recalling \ref{ass:g} and using again that $\ph_\l(0)=-1+\l$, the formal integral is rigorously defined as follows
\begin{align*}
\luca{-\int_0^t\iO g(s) \dt \phl(s)\,\ds}
& \luca{=-\int_0^t\langle\partial_t\ph_\lambda(s), g_0+g_1(s)\rangle_{V}\,\ds
-\int_\Omega g_2(t)\varphi_\lambda(t)}
\\ &\quad \quad
\luca{- \int_\Omega g_2(0)(1-\l)
+\int_0^t\iO \partial_tg_2(s) \phl(s)\,\ds}.
\end{align*}
By comparison in \eqref{SYS:1:app}, using the boundedness of $S$,
\ref{ass:g}, and \eqref{prelim:reg:local}, thanks to the Young inequality, we infer that
\[
  \luca{-\int_0^t\iO g(s) \dt \phl(s)\,\ds}
  \leq\frac14\int_0^t \norma{\nabla \ml(s)}^2\,\ds + C.
\]
Moreover, using again the boundedness of $S$, \eqref{SYS:2:app}, and
\eqref{prelim:reg:local}, we find
\begin{align*}
	& \int_0^t\iO S(\cdot,\phl(s),g(s)) \ml(s)\,\ds\\
	&\quad= \int_0^t\iO S(\cdot,\phl(s),g(s)) (\ml(s)- (\ml(s))_\Omega) \,\ds
	+ \int_0^t\iO S(\cdot,\phl(s),g(s)) (\ml(s))_\Omega \,\ds\\
	&\quad\leq
	\frac 14 \iot \norma{\nabla \ml(s)}^2 \,\ds
	+ C
	+ C \iot |(\ml(s))_\Omega|.
\end{align*}
Observe now that, by comparison in equation \eqref{SYS:2:app}, \luca{thanks to
{\bf A2} and \eqref{prelim:reg:local}, it holds
\begin{align*}
	|\Omega||(\ml)_\Omega|
	& \leq \norma{\b(\phl)}_1 + \norma{\pi(\phl)}_1 + \norma{g}_1
	\leq \norma{\b(\phl)}_1 + C .
\end{align*}
Putting everything together, we infer that there exists a constant $C_*>0$,
independent of $\lambda$, such that, for every $t\in[0,T]$,
\begin{align}
	 \norma{\nabla \phl(t)}^2  +
	 \rev{\|{\widehat\beta(\varphi_\lambda(t))}\|_1}
	+\int_0^t \norma{\nabla \ml(s)}^2\,\ds
	&\leq C_* + C_*\int_0^t\|\beta(\varphi_\lambda(s))\|_{1}\ds\rev{.}
	\label{est:energy:1'}
\end{align}
}
\luca{Now, we go back to equation \eqref{SYS:2:app} tested by $\phl - (\phl)_\Omega$
and get
\begin{align*}
	& \norma{\nabla\phl}^2+
	\int_\Omega\beta(\phl)(\phl-(\phl)_\Omega)	
	\leq
	\iO \ml 	\big(\phl - (\phl)_\Omega\big)
	 	 - \iO (\pi(\phl)+g)	\big(\phl - (\phl)_\Omega\big).
\end{align*}
By exploiting the fact that
$\varphi_{0,\lambda}=(\varphi_{0,\lambda})_\Omega=-1+\lambda$,
using the Poincar\'e--Wirtinger inequality, and assumptions {\bf A1--A2},
one has that
\begin{align*}
	&
	\iO \ml 	\big(\phl- (\phl)_\Omega\big)
	 - \iO (\pi(\phl)+g)	\big(\phl- (\phl)_\Omega\big)
	 \\& \quad
	 =
	 \iO (\ml - (\ml)_\Omega) 	\big(\phl- (\phl)_\Omega\big)
	 	- \iO (\pi(\phl) + g)	\big(\phl- (\phl)_\Omega\big)
	 \\ & \quad
	 =
	\iO (\ml - (\ml)_\Omega) 	\big((\phl-\ph_{0,\l}) - ( \phl - (\ph_{0,\l} ))_\Omega\big)
	 	\\ & \qquad
	 	- \iO (\pi(\phl) + g)	\big((\phl-\ph_{0,\l}) - ( \phl - (\ph_{0,\l} ))_\Omega\big)
	 \\ & \quad
	 \leq C \norma{\nabla\ml}\norma{\phl- \ph_{0,\l}}
	 + C(\norma{\phl}+1)\norma{\phl- \ph_{0,\l}}.
\end{align*}
\mod{Here we have used the identity $\iO (\ml)_\Omega 	\big(\phl- (\phl)_\Omega\big)=(\ml)_\Omega \iO 	\big(\phl- (\phl)_\Omega\big)=0$ as well as the fact that \(\ph_{0,\l}\) is constant and, therefore, coincides with its mean value \((\ph_{0,\l})_\Omega\).}
Hence, thanks to the estimate \eqref{prelim:reg:local} we deduce that
\begin{align}
\label{est:energy:2'}
	\int_\Omega\beta(\phl(s))(\phl(s)-(\phl(s))_\Omega)	
	\leq C (1+\norma{\nabla\ml(s)})\norma{\phl(s)- \ph_{0,\l}}
	\quad\forall\,s\in[0,T].
\end{align}
}

\luca{Next, we note that \eqref{claim:mean} and \eqref{claim:mean2} imply that
there exists $T_0\in(0,T)$ and $r_0\in(0,1)$, both independent of $\lambda$, such that,
for every $\lambda\in(0,\lambda_0)$, it holds that
$-1<(\varphi_\lambda(s))_\Omega<-1+r_0$
for all $s\in(0,T_0)$ and $-1+r_0\leq(\varphi_\lambda(s))_\Omega\leq1-r_0$
for all $s\in[T_0,T]$. Consequently, for $s\in(0,T_0)$
we can make use of estimate \eqref{MZ:sharp'} with $\phi=\phl(s)$
and infer that
\begin{align}
\nonumber
	  &c_\b ((\phl(s))_\Omega+1) \int_\Omega|\b(\phl(s))|\\
\nonumber
	  & \quad \leq
  \int_\Omega\beta(\phl(s))(\phl(s)-(\phl(s))_\Omega)
  \\ & \qquad
  \label{ineq_luca1}
  +C_\beta((\phl(s))_\Omega+1)
  |\b(-1+((\phl(s))_\Omega+1)/2)|
  \qquad\forall\,s\in(0,T_0).
\end{align}
Analogously, for $s\in[T_0,T]$ we can make use
of estimate \eqref{MZ:std}
and infer that
\begin{align}
\label{ineq_luca2}
	  c_* \int_\Omega|\b(\phl(s))|
	  & \leq
  \int_\Omega\beta(\phl(s))(\phl(s)-(\phl(s))_\Omega)
  +C_*|\Omega| \qquad\forall\,s\in[T_0,T].
\end{align}
Consequently, we can combine \eqref{est:energy:2'} evaluated at $s\in(0,T_0)$
with inequality \eqref{ineq_luca1}, divide by $(\phl(s))_\Omega+1$,
and exploit the estimates \eqref{est:alpha:local} and \eqref{claim:mean2}
to get, for all $s\in(0,T_0)$ and $\alpha\in[0,\frac34]$
(resp.~$\alpha\in(0,1)$ if $g_0\in W$),
\begin{align}
\nonumber
  c_\beta\|\beta(\varphi_\lambda(s))\|_1 &\leq
  C\frac{\norma{\phl(s)- \ph_{0,\l}}}{(\phl(s))_\Omega+1}
  (1+\norma{\nabla\ml(s)}) +
  C_\beta|\b(-1+((\phl(s))_\Omega+1)/2)|\\
\label{ineq_luca1'}
  &\leq \frac{CC_\alpha}{2c_1}s^{\alpha-1}(1+\norma{\nabla\ml(s)})
  +C_\beta|\b(-1+c_1s)|.
\end{align}
Analogously, if we combine \eqref{est:energy:2'} evaluated at $s\in[T_0,T]$
with inequality \eqref{ineq_luca1} and estimate \eqref{prelim:reg:local}, we get,
for all $s\in[T_0,T]$,
\begin{align}
\label{ineq_luca2'}
  c_*\|\beta(\varphi_\lambda(s))\|_1 \leq
  C(1+\norma{\nabla\ml(s)})+
  C_*|\Omega|.
\end{align}
Collecting the above estimates, integrating with respect to $s\in(0,t)$, and using the Young inequality,
for all $\delta>0$, we find a constant $C_\delta>0$, independent of $\lambda$,
such that, for every $t\in[0,T]$,
\begin{align*}
  \int_0^t\|\beta(\varphi_\lambda(s))\|_1\ds
  \leq \delta\int_0^t\norma{\nabla\ml(s)}^2\ds+C_\delta+
  C_\delta\int_0^{T_0}s^{2(\alpha-1)}\ds +
  C\int_0^{T_0}|\b(-1+c_1s)|\ds.
\end{align*}
As for the right-hand side, the last term
makes sense by {\bf A1} as
\begin{align*}
  \int_{0}^{T_0}|\beta({-1 + c_1s)}|\,\ds
  =\frac1{{c_1}}\int_{-1 }^{-1+ c_1 T_0}
  |\beta(s)|\,\ds\leq
  C\norma{\beta}_{L^1(-1,0)} \leq C.
\end{align*}
By choosing then $\alpha\in(\frac12,\frac34]$
(resp.~$\alpha\in(\frac12,1)$ if $g_0\in W$),
we infer then that
\begin{align}
\label{est:energy:3'}
  \int_0^t\|\beta(\varphi_\lambda(s))\|_1\ds
  \leq \delta\int_0^t\norma{\nabla\ml(s)}^2\ds+C_\delta
  \quad\forall\,t\in[0,T].
\end{align}
}

\luca{Eventually, recalling that $C_*$ is the constant appearing in \eqref{est:energy:1'},
we multiply \eqref{est:energy:3'} by $2C*$ and sum it with
inequality \eqref{est:energy:1'}. By choosing $\delta$ small enough,
for instance $\delta:=\frac1{4C_*}$, and rearranging the terms, we deduce the bound}
\begin{align}
	&
	\norma{\phl}_{\L\infty V}
	+ \|\hat \b(\phl)\|_{\L\infty {\Lx1}}
	+ \| \b(\phl)\|_{\L 1{\Lx1}}
	+ \norma{\nabla \ml}_{\L2 H}
	\leq C.
\label{en:est:local}
\end{align}
From this, it is a standard matter to derive by comparison in equation \eqref{SYS:1:app} that
\begin{align}
	\label{est:dt}
	\norma{{\dt}\phl}_{{\L2 \Vp}}
	\leq C.
\end{align}
Once the above is established, we go back to
\luca{\eqref{ineq_luca1'}--\eqref{ineq_luca2'},
take square powers, and exploit the fact that now
$t \mapsto \norma{\nabla \ml(t)}$ is bounded in $L^2(0,T)$.
To this end, we note that by {\bf A1} one has
\begin{align*}
  \int_{0}^{T_0}|\beta({-1 + c_1s)}|^2\,\ds
  =\frac1{{c_1}}\int_{-1 }^{-1+ c_1 T_0}
  |\beta(s)|^2\,\ds\leq
  C\norma{\beta}_{L^2(-1,0)}^2 \leq C,
\end{align*}
while for every $\ell\in(1,4)$ (resp.~for every $\ell\in(1,+\infty)$ if $g_0\in W$)
there exists $\alpha\in (0,\frac34]$ (resp.~$\alpha\in (0,1)$ if $g_0\in W$)
such that ${\ell}(\alpha-1)>-1$.}
Hence, for every such $\ell$,
there exists also a constant $C_\ell>0$, independent of $\lambda$, such that
\[
  \norma{t\mapsto {t^{\alpha -1}}}_{L^\ell(0,T_0)}\leq C_\ell.
\]
Thus, going back to \eqref{ineq_luca1'}--\eqref{ineq_luca2'}, we can now infer that
\begin{align*}
\forall\,p& \in(1,\ov p), \quad\exists\,C_p>0:\quad	
\luca{\norma{\beta (\phl)}_{L^p(0,T; L^1(\Omega))}
	\leq C_p,}
	\end{align*}
with $\ov p$ being defined in \eqref{def:bp}.
Next, comparison in equation \eqref{SYS:2:app} readily produces
\begin{align*}
\forall\,p& \in(1,\ov p), \quad\exists\,C_p>0:\quad	\norma{(\ml)_\Omega}_{L^p(0,T)}
	\leq C_p,
	\end{align*}
whence, using the Poincar\'e inequality, that
\begin{align}
	\label{chem:est}
\forall\,p& \in(1,\ov p), \quad\exists\,C_p>0:\quad
	\norma{\ml}_{\L p {V}}
	\leq C_p.
\end{align}
\luca{By monotonicity and comparison in \eqref{SYS:2:app} this leads
by classical techniques also that
\begin{align}
	\label{beta:est}
\forall\,p& \in(1,\ov p), \quad\exists\,C_p>0:\quad
	\norma{\beta(\varphi_\lambda)}_{\L p {H}}
	\leq C_p.
\end{align}
Lastly}, arguing  as in \cite[Lemmas 7.3-7.4]{GGW} and using interpolation, we read
\eqref{SYS:2:app} as a one-parameter family of time-dependent system of elliptic equation with maximal monotone nonlinearity $\b$ and infer that
\begin{align}
\label{est17}
  \forall\,p& \in(1,\ov p), \quad\exists\,C_p>0:\quad
  \norma{\phl}_{L^{2p}(0,T;H^2(\Omega)) \cap L^p(0,T; W^{2,\tre}(\Omega))}
	\leq C_p,
\end{align}
with $\gamma$ being defined in \eqref{def:bp}.

\noindent
{\bf Nonlocal case:}
We can argue as above.
Again, we test \eqref{SYS:1:app:nonloc} by $\ml$, \eqref{SYS:2:app:nonloc} by $- \dt \phl$ and add the resulting equalities. This produces
\begin{align*}
	&\frac 14\int_{\Omega \times \Omega}
	K(x-y)|\phl (t,x)- \phl(t,y)|^2 {\rm dx dy}  + \iO F(\phl(t))
	+\int_0^t \norma{\nabla \ml(s)}^2\,\ds \\
	&\quad=|\Omega|F(-1+\l)+
	\int_0^t\iO S(\cdot,\phl(s),g(s)) \ml(s)\,\ds
	- \int_0^t\iO g(s) \dt \phl(s)\,\ds,
\end{align*}
where the last term on the right-hand side
has to be interpreted rigorously as before. On account of \rev{\ref{ass:kernel}}, we have
\begin{align*}
	\int_{\Omega \times \Omega}
	K(x-y)|\phl (t,x)- \phl(t,y)|^2 {\rm dx dy}
	\leq 2|\Omega|a^*\norma{\phl - \ph_{\l,0}}_{L^\infty(0,T; H)}^2.
\end{align*}
Thus, reasoning as above and exploit\mod{ing} \eqref{prelim:reg:nonlocal}, we find that, for any $p\in(1,2)$, there exists
$C_p>0$ such that
\begin{align}
	\non
	&\norma{	\dt \phl}_{\L2 \Vp}
	+ \norma{\phl}_{\L\infty H}
		+ \|\hat \b(\phl)\|_{\L\infty {\Lx1}}
	+ \luca{\| \b(\phl)\|_{\L p{H}}}\\
	\label{en:est:nonlocal}
	&+\norma{\ml}_{\L p V}
		+ \norma{\nabla \ml}_{\L2 H} \leq C_p.
\end{align}

\subsection{Step (V)}
The final step consists in passing to the limit as $\l \to 0$ along a suitable subsequence.
Thus, everything that we are going to present from now on has to be in principle considered for a suitable subsequence.
However, in order to keep notation as light as possible,
we stipulate that every passage to the limit has to be intended this way.

\noindent
{\bf Local case.}
Accounting for \eqref{prelim:reg:local} and  \mod{\eqref{est:dt}},
we infer the existence of a function $\ph$ such that
\begin{align*}
	\ph &  \in \H1 {\Vp} \cap \L\infty V \cap \L2 {H^2(\Omega)},
\end{align*}
and that, as $\l \to 0$,
\begin{align*}
	\ph_\l &  \to \ph \quad \luca{\text{weakly in} \quad  \H1 {\Vp}  \cap
	\L2{H^2(\Omega)}},\\
	\ph_\l &  \to \ph \quad \luca{\text{weakly-star in} \quad   \L\infty V }\rev{.}
\end{align*}
The classical Aubin--Lions theorem then yields that, as $\l\to 0$,
\begin{align*}
		\ph_\l &  \to \ph \quad \text{strongly in} \quad  \C0 H \cap \L2 V
\end{align*}
so that the initial condition $\ph(0)=-1$ is fulfilled.
Our final goal is to pass to the limit, as $\l \to 0$, in the weak formulation associated to \eqref{SYS:1:app}--\eqref{SYS:4:app} to show
that  $\phl$ and $ \ml$ converge to some \rev{limit} $\ph$ and $\m $ that yield a solution to the weak formulation of problem \eqref{sys:1}--\eqref{sys:4}.
\luca{As far as $\beta$ is concerned, estimate \eqref{beta:est}, the strong
convergence of $\varphi_\lambda$, and the strong-weak
closure of $\beta$ as a maximal monotone graph in $\mathbb R\times\mathbb R$ yield that
\[
  \beta(\ph_\l)   \to \beta(\ph) \quad \text{weakly in} \quad  L^p(0,T; H)
  \quad\forall\,p\in(1,\ov p).
\]
Let us stress that this \rev{latter}   holds along a suitable subsequence
for every $p\in(1,\ov p)$, in the sense that the choice of the subsequence
can be done uniformly with respect to $p\in(1,\ov p)$.
This fact can be shown rigorously in the following way.
First, by exploiting the estimate \eqref{beta:est}
with \mod{$p \in (1,\ov p)$  (e.g., $p=\frac {\ov p+1} 2$}),
one can select a suitable subsequence along which it holds that
$\beta(\varphi_\lambda)\to\beta(\varphi)$ weakly in $L^{\frac{\ov p}2}(0,T; H)$.
Secondly, let us fix such subsequence and
prove that along such fixed subsequence it also holds that
$\beta(\varphi_\lambda)\to\beta(\varphi)$ weakly in $L^p(0,T; H)$
for all $p\in(1,\ov p)$. To this end, let $p\in(1,\ov p)$ be arbitrary
and let
$(\beta(\varphi_{\lambda_k}))_k$ be an arbitrary subsequence of
$(\beta(\varphi_{\lambda}))_\lambda$. Thanks to \eqref{beta:est}
it follows immediately that $(\beta(\varphi_{\lambda_k}))_k$ has a further
sub-subsequence $(\beta(\varphi_{\lambda_{k_j}}))_j$
weakly converging to some limit in $L^p(0,T; H)$ \rev{as $j\to0$}. As we know already that
$\beta(\varphi_\lambda)\to\beta(\varphi)$ weakly in $L^{\frac{\ov p}2}(0,T; H)$,
it necessarily holds that $\beta(\varphi_{\lambda_{k_j}})\to\beta(\varphi)$
weakly in $L^p(0,T; H)$ \rev{as $j\to0$}. This shows that every subsequence of
$(\beta(\varphi_{\lambda}))_\lambda$ admits a further sub-subsequence
weakly converging in $L^p(0,T; H)$ to $\beta(\varphi)$: this implies
that the whole sequence $(\beta(\varphi_{\lambda}))_\lambda$ weakly converges
to $\beta(\varphi)$ in $L^p(0,T; H)$  \rev{as $\l\to0$}, as desired.
As} for $\pi(\phl)$,  we use the above strong convergence of the order parameter, along with its \Lip\ continuity, to infer that
\begin{align*}
		\pi(\ph_\l)   \to \pi(\ph) \quad \text{strongly in} \quad  \C0 H\rev{,}
\end{align*}
and (see \eqref{source:splitting1})
\[
  m{(\cdot)}\phl   \to m{(\cdot)}\ph \quad \text{strongly in} \quad  \C0 H.
\]
Besides, using the continuity of $h$ with respect to its first argument, the boundedness of $h$,
and the dominated convergence theorem, we get that, as $\l \to 0$,
\begin{align*}
	S(\cdot,\phl,g) \to S(\cdot,\ph,g) \quad \text{strongly in} \quad  \L2 H.
\end{align*}
For the chemical potential we owe to \eqref{chem:est}  to infer that, as $\lambda\to0$,
\begin{alignat*}{2}
  \ml & \to\mu \quad&&\text{weakly in }
  L^p(0,T;V) \cap L^2(\sigma,T;V)
  \qquad\forall\,p\in(1,\ov p),\,\forall\,\sigma\in(0,T),
	\\
  \nabla\mu_\lambda &\to \nabla\mu \quad&&\text{weakly in } L^2(0,T;H).
\end{alignat*}
Exploiting again the strong convergence of $\phl$, we find from \eqref{est17} that
\begin{align}
	\ph \in  &\bigcap_{p\in(1,\ov p)}
	L^{2p}(0,T;H^2(\Omega)) \cap L^p(0,T; W^{2,\tre}(\Omega)),
	\label{loc:reg:ell}
\end{align}
with $\tre$ as in the statement. Observe that, arguing as above,
one can actually select the converging subsequences uniformly with respect to
$p$ and $\sigma$.
Next, we use all the above convergences to pass  to the limit $\l \to 0$ in the variational formulation associated to \eqref{SYS:1:app} and find that
\begin{align*}
	& \<\dt \ph ,v>_{V}
	+ \iO \nabla \mu \cdot \nabla v = \iO {S(\cdot,\ph,g)} v,
\end{align*}
for every test function $ v\in V$, and almost everywhere in $(0,T)$.
Testing \eqref{SYS:2:app} for an arbitrary test function $ v\in V$ and integrating by parts leads us to
\begin{align*}
		\iO  \ml v
		= \iO  \nabla \ph_\lambda \cdot \nabla v
		+ \iO  \b(\ph_\lambda) v
		+ \iO  \pi(\ph_\lambda)v
		+ \iO  g v,
		\quad
		\text{a.e. in $(0,T)$.}
\end{align*}
Exploiting the above convergences, we pass to the limit as $\l \to 0$ and get the boundary condition
\begin{align*}
	\dn \ph =0, \quad \text{a.e. on $\Sigma$,}
\end{align*}
along with the pointwise formulation
\begin{align*}
	 \mu = -\Delta \ph +  \b(\ph)+ \pi(\ph) +  g,
	\quad
	\text{a.e. in $Q$.}
\end{align*}
Thus $(\ph,\m)$ is a variational solution to the original problem \eqref{sys:1}--\eqref{sys:4} \luca{and the proof of Theorem~\ref{THM:EX:LOCAL} is finished.}

\noindent
{\bf Nonlocal case.}
The situation is similar to the previous case, so we just give a sketch of the proof.
From \eqref{prelim:reg:nonlocal} and \eqref{en:est:nonlocal} we obtain the existence of a function $\ph$ such that
\begin{align*}
	\ph \in  \H1 \Vp \cap \L\infty H \cap \L2 V,
\end{align*}
and, as $\l \to 0$,
\begin{alignat*}{2}
	\ph_\l &  \to \ph \quad &&\luca{\text{weakly in}
	\quad   \H1 \Vp \cap \L2 V,}
	\\
	\ph_\l &  \to \ph \quad &&\luca{\text{weakly-star in}
	\quad \L\infty H,}
	\\
	\ph_\l &  \to \ph \quad &&\text{strongly in} \quad \C0 H,
\end{alignat*}
which yields that the initial condition $\ph(0)=-1$ is fulfilled in $H$.
Here we have lower regularity with respect to the previous case. Nonetheless, the above strong convergence is enough to handle the nonlinear terms in the same way.
The only difference concerns the convolution kernel for which we have, on account of \ref{ass:kernel}, that, as $\l \to 0$,
\begin{align*}
		K * \ph_\l &  \to K * \ph \quad \text{strongly in} \quad  \L2V.
\end{align*}
The rest of the proof goes as above and no repetition is needed here.
The only difference is that, here, we can no longer infer the regularity properties \eqref{loc:reg:ell}.
\luca{This concludes the proof of Theorem~\ref{THM:EX:NONLOCAL}.}

\section{Uniqueness of weak solutions}
\label{uniqproof}
We will now proceed to demonstrate Theorems \ref{THM:UNI:LOCAL} and \ref{THM:UNI:NONLOCAL}. We will give the full details of the former proof only, since the latter proof goes along the same lines.

\begin{proof}[Proof of Theorem \ref{THM:UNI:LOCAL}]
First, we set $\ph=\ph_2 -\ph_1$ and consider the weak formulation for the difference that reads\mod{, recalling that $h$ is constant,} as
\begin{align}\label{cd:first:wf}
	& \<\dt \ph ,v>_{V}
	+ \iO \nabla \mu \cdot \nabla v = - \iO m(\cdot)\ph v \mod{,}
\end{align}
for every test function $ v\in V$, and almost everywhere in $(0,T)$, where
\begin{align*}
		\mu=  - \Delta \ph + F'(\ph_2) - F'(\ph_1), \quad & \text{$a.e.$ in $Q$}.
\end{align*}
Taking $v=|\Omega|^{-1}$ in  \eqref{cd:first:wf}, we get
\begin{align}	\label{mean:cd}
	\frac d {dt} \ph_\Omega
	= 	- (m (\cdot)\ph)_\Omega\mod{.}
\end{align}
Multiplying now the above equation by $v=\ph_\Omega + \sign \ph_\Omega$\mod{,} we obtain
\begin{align}	\label{cd:mean2}
	 \frac d{dt} \Big(\frac 12|{\ph_\Omega} |^2 + |\ph_\Omega|\Big)\leq C (|\ph_\Omega|^2 \mod{+|\ph_\Omega| }).
\end{align}
We then subtract equation \eqref{mean:cd} from \eqref{cd:first:wf}, test the resulting identity with $v=\NN(\ph- \ph_\Omega)$ and use the equation for the chemical potential to express the term $\iO \nabla \mu \cdot \nabla \NN(\ph- \ph_\Omega)$.
Upon rearranging some terms, this gives, for almost any $t\in (0,T)$,
\begin{align}	\non
	&\frac12\frac d {dt} \norma{\ph-\ph_\Omega}^2_{*}
	+ \norma{ \nabla \ph }^2  + \iO (\beta(\ph_2) - \beta(\ph_1))\ph\\
    &\non \quad = -\iO m(\cdot)\ph \NN(\ph-\ph_\Omega)
        - \iO (\pi(\ph_2) - \pi(\ph_1))(\ph- \ph_\Omega)
\\ & \qquad
    - \iO (\beta(\ph_2) - \beta(\ph_1))\ph_\Omega=: \sum_{i=1}^\mod{3} I_i.
    \label{cd:proof:est:1}
\end{align}
The third term on the \lhs\ is nonnegative due to the monotonicity of $\beta$, whereas the terms on the \rhs\ can be bounded as follows.
Recalling the interpolation inequality:
\begin{align*}
	\forall \delta >0,\quad \exists \,\cd>0 : \quad
	\norma{\ph}^2
	\leq \delta \norma{\nabla \ph}^2
	+ \cd \norma{\ph-\ph_\Omega}^2_*
	+ C |\ph_\Omega|^2,
\end{align*}
and using the Young inequality and the above estimate, we infer that
\begin{align*}
	I_1 + \mod{I_2 } & \leq
	C \norma{\ph}^2
	+ C\norma{\ph-\ph_\Omega}^2_*
	\leq
	\delta \norma{\nabla \ph}^2
	+ \cd \norma{\ph-\ph_\Omega}^2_*
	+ C |\ph_\Omega|^2,
	\\
	\mod{I_3} & \leq
	C (\norma{\beta (\ph_2)}_1 + \norma{\beta (\ph_1)}_1)|\ph_\Omega|.
\end{align*}
Next, we add \eqref{cd:proof:est:1} with  \eqref{cd:mean2} and observe that $t \mapsto C (\norma{\beta (\ph_2(t))}_1 + \norma{\beta (\ph_1(t))}_1)$ belongs to $L^1(0,T)$. Thus, integrating over time, choosing $\delta$ small enough, and applying Gronwall's lemma we deduce  \eqref{loccd}. This concludes the proof.
\end{proof}
\begin{proof}[Proof of Theorem \ref{THM:UNI:NONLOCAL}]
Here, we just emphasize the key differences with respect to the previous proof.
\mod{Also in this case, setting $\ph=\ph_2 -\ph_1$, we have (cf. \eqref{cd:first:wf})}
\begin{align*}
	& \<\dt \ph ,v>_{V} + \iO \nabla \mu \cdot \nabla v =
-\iO m(\cdot)\ph v \mod{,}
\end{align*}
for every test function $ v\in V$, and almost everywhere in $(0,T)$, where now
\begin{align}
\label{nlchempot}
	\mu=  a\ph - K* \ph + F'(\ph_2) - F'(\ph_1), \quad & \text{$a.e.$ in $Q$}.
\end{align}
Arguing as in the proof above, we derive
\eqref{mean:cd} and  \eqref{cd:mean2}.
We then repeat the previous computation with the test function $v=\NN(\ph- \ph_\Omega)$, using \eqref{nlchempot}.
Thus, for almost any $t\in (0,T)$, we get
\begin{align*}	\non
	&\frac 12 \frac d {dt} \norma{\ph-\ph_\Omega}^2_{*}
	+ \norma{  \ph }^2
     \leq \iO (K*\ph)(\ph-\ph_\Omega)
    +\iO a\ph(\ph-\ph_\Omega)
   \\ & \quad
    -\iO m(\cdot)\ph \NN(\ph-\ph_\Omega)
    \mod{- \iO (\pi(\ph_2) - \pi(\ph_1))(\ph- \ph_\Omega)}\\ & \quad
    - \iO (\beta(\ph_2) - \beta(\ph_1))\ph_\Omega,
\end{align*}
where the second term on the \lhs\ arises from $ \iO (a\ph + \beta(\ph_2)
- \beta(\ph_1))\ph$ using \ref{ass:pot}, \mod{and} \ref{ass:kernel}. Observe now that, using the Young inequality, we have
\begin{align*}
	\iO (K*\ph)(\ph-\ph_\Omega)
	& \leq
	\norma{K*\ph}_V\norma{\ph-\ph_\Omega}_*\leq\norma{K}_{W^{1,1}(\Omega)}\norma{\ph}\norma{\ph-\ph_\Omega}_*\\ & \leq \delta\norma{\ph}^2
	+\cd\norma{\ph-\ph_\Omega}_*^2,
\end{align*}
for a positive $\delta$ yet to be selected.
The remaining terms on the right-hand side can be readily controlled again by the Young inequality. Thus, we can proceed arguing as above
to get \eqref{nonloccd}.
\end{proof}

\section{\rev{A}pplication to tumor growth models}
\label{SEC:TUMOR}

This section is devoted to a nontrivial application of the existence results analyzed above in the context of tumor growth models.
The diffuse interface approach is nowadays broadly used in biological modeling, in particular, to describe (solid) tumor growth dynamics (see, e.g., \cite{CL, GLSS}).
The prototypical scenario that they try to capture may be schematized as follows: a tumor mass is embedded in a region rich of nutrient and it expands through a proliferation mechanism which is possible due to the consumption of the surrounding nutrient.
In the literature, there are several models, but their common denominator is a coupling between a Cahn--Hilliard equation with a source term, accounting for the evolution of the tumor mass,
with a reaction-diffusion equation for some nutrients species $n$. Therefore, problems \eqref{sys:1}--\eqref{sys:4} and \eqref{sys:non:1}--\eqref{sys:non:4} are related to those models.
The role of the source term is to encapsulate nontrivial interactions between the tumor cells and the nutrient such as the proliferation mechanism.
More precisely, in this context, the order parameter $\ph$ describes the presence of tumor cells with the convention that the level set $\{\ph=-1\}:= \{x \in \Omega: \ph(x)=-1\}$ describes the healthy region, whereas $\{\ph=1\}$ the tumor region.
The free energies \eqref{energy:local} and \eqref{energy:nonlocal} account for cell-to-cell adhesion effects, which may be of local or nonlocal nature, respectively.

Let us now introduce some meaningful models which fit the framework of the Cahn--Hilliard equations analyzed so far.
Usually, the evolution of the nutrient $n$ is ruled by a reaction-diffusion equation of the form
\begin{align}	\label{nutrient:abs}
		\dt n - \div \big(\eta(\ph)\nabla n \big) = S_n(\ph,n) \quad \text{in $Q$,}
\end{align}
endowed with no-flux or Robin boundary conditions as well as a suitable initial condition.
Here, $S_n(\ph,n)$ stands for a nutrient source term, and $\eta(\ph)$ for the nutrient mobility.
For more details on the modeling we refer, for instance, to \cite{GLSS} and references therein.

Thus, we can introduce the following initial and boundary value problems for tumor growth systems whose well-posedness readily frame in the setting of the above results
\begin{alignat}{2}	
	\label{eq:1:abs:tumor}
	& \dt \ph - \Delta (A_j \ph + F'(\ph) - \chi n) = S(\ph,n) \qquad &&\text{in $Q$,} \quad j=1,2,
	\\
	\label{eq:2:abs:tumor}
	& \dt n - \Delta n = S_n(\ph,n)\quad && \text{in $Q$,}
	\\
	\label{eq:3:abs:tumor}
	& \dn (A_j \ph + F'(\ph) - \chi n)= \dn n=  0 \qquad &&\text{on $\Sigma$,}
	\\
	\label{eq:4:abs:tumor}
	& \ph(0) = -1, \quad n(0)= n_0 \qquad &&\text{in $\Omega$,}
\end{alignat}
where $A_1$ is the usual (minus) Laplace operator with homogeneous Neumann boundary condition, while
$A_2$ is the linear nonlocal operator from $H$ to $H$ defined by \rev{$A_2 \ph = a \ph -K*\ph$}.
Besides, $\chi\geq  0$ denotes a nonnegative constant related to the chemotaxis sensitivity.
Typical choices for the source terms $S$ and $S_n$ are the following (see, e.g., \cite{ALL, AS}):
\begin{align}
	\label{defsource}
	S(\ph,n)=- m \ph + (n - \d_n)_+ \h(\ph),
	\quad
	S_n(\ph,n)=  {\cal B} (n_B - n) - {\cal C}   \h(\ph)n ,
\end{align}
for prescribed constants $m >0,$ and $\d_n \in [0,1]$. From now on, we postulate \eqref{defsource} for the source terms.
The capital letters ${\cal B}$ and ${\cal C}$ indicate nonnegative constants taking into account
consumption rate of the nutrient with respect to a pre-existing concentration $n_B$,
and the nutrient consumption rate, respectively.
Here,  $\h$ plays the role of a truncation so to have the action of ${\cal C} \h(\ph)n $ only in the spatial region where the tumor cells are located\rev{. Namely, we postulate} that
$\h(-1) = 0$, $\h(1) = 1$, and $\h$ is suitably interpolated in between if $-1< \ph < 1$.
Finally, the source term ${\cal B} (n_B - n) $ in \eqref{eq:2:abs:tumor} models the nutrient supply from the blood vessels if $n_B > n$ and the nutrient transport away from the
domain for $n_B< n$.

The novelty of our approach consists in enabling the initial condition $\ph_0$ to be chosen as the pure healthy phase.
We somehow show that even if  the tissue of the patient is completely  healthy at the beginning\rev{,} corresponding to $\ph_0=-1$, the presence of the disease
(captured by the structure of the source term) might trigger the evolution and the growth of the tumor.
Below, we provide existence results for the above local and nonlocal problems where the initial datum $\ph_0 $ can correspond to the healthy state, i.e., $\ph_0=-1$ in $\Omega$.

With no claim to be exhaustive, let us briefly review some related works dealing with singular potentials like \eqref{Flog}.
For those models, a crucial point is to handle the mass evolution; in particular one would avoid inconsistency of the model and have non physical solutions $\ph$ that satisfy $|(\ph(t))_\Omega | >1$ for some time $t \in[0,T]$.
To overcome this possible situation one usually proceed as we did and requires that the source term $S$, which rules the mass evolution of the order parameter, is given by \eqref{defsource}.
In this direction, we refer to \cite{GLRS, RSchS, FLRS}.
Alternatively, some authors suggest to introduce some regularization terms in the system to allow for more general potentials and source terms (see, for instance, \cite{CGH,SS} and the references therein).
Finally, let us point out that nonlocal models in the context of phase field tumor growth models are also investigated as well  (see, e.g., \cite{SS, FLS}).

As a direct consequence of the previous sections, we can deduce the validity of the following weak existence results.
\begin{theorem}\label{THM:TUMOR:LOCAL}
Suppose that \ref{ass:pot}--\ref{ass:S} hold.
Let the initial data fulfill
\begin{align*}
	& \ph_0 \in V,
	\quad
	F(\ph_0) \in \Lx1,
	\quad
	(\ph_0)_\Omega \in [-1,1],
	\quad n_0 \in V\cap \Lx\infty,
	\\
	&
	0 \leq n_0(x) \leq 1 \quad \text{for a.e.} x \in \Omega,
\end{align*}
and assume that $\h$ is  \anlast{continuous and} uniformly bounded, $\chi$ is a nonnegative constant,
and $n_B \in L^\infty(Q)$ with $0 \leq n_B \leq 1$ a.e. in $Q$.
Then, problem \eqref{eq:1:abs:tumor}--\eqref{eq:3:abs:tumor} with $j=1$
admits a weak solution $(\ph, \mu, n)$ such that
\begin{align*}
	& \ph \in \H1 \Vp \cap \L\infty V \cap \L2 W,
	\\
	& \ph \in \Ls4 W \cap \Ls2 {W^{2,\tre}(\Omega)} \quad \forall \sigma \in (0,T),
	\\
	& \ph \in \L {2p} W \cap \L {p} {W^{2,\tre}(\Omega)} \quad \forall p\in (1,\ov p),
	\\
	& \ph \in L^\infty(Q), \quad |\ph(x,t)| <1 \quad\text{ for a.a. } (x,t) \in Q,
	\\
	& \mu \in \L p  V \cap \Ls2 V \quad \forall p\in (1,\ov p) \quad  \forall \sigma \in (0,T),
    \\
    &\nabla \mu \in \L2 H,
    \\
    & n \in \H1 H \cap \L\infty V \cap \L2 W,
    \\
     & n \in L^\infty(Q), \quad 0 \leq n(x,t)\leq  1 \quad\text{ for a.a. } (x,t) \in Q,
\end{align*}
where
\[
\tre\in\begin{cases}
[1,6] \quad&\text{if } d=3,\\
[1,+\infty) \quad&\text{if } d=2,
\end{cases}
\qquad\qquad
\ov p:=
\begin{cases}
  \frac43 \quad&\text{if } n_0\in V\setminus W,\\
  2\quad&\text{if } n_0\in W,
\end{cases}
\]
where
\begin{align*}
		\mu=  - \Delta \ph + F'(\ph) - \chi n, \quad & \text{$a.e.$ in $Q$},
\end{align*}
and the following identities
\begin{align*}
	& \<\dt \ph ,v>_{V}
	+ \iO \nabla \mu \cdot \nabla v = \iO S(\ph,n) v,
	\\
	&\iO \dt n \, v
	+  \iO \nabla n \cdot \nabla v
	+ \iO \h(\ph) n v
	= {\cal B} \iO (n_B- n) v,
\end{align*}
are satisfied for every test function $ v\in V$, and almost everywhere in $(0,T)$.
Moreover, it holds that
\begin{align*}
	\ph(0)=\ph_0, \quad n(0)=n_0, \quad \text{a.e. in } \Omega.
\end{align*}
\end{theorem}

\begin{theorem}\label{THM:TUMOR:NONLOCAL}
Suppose that \ref{ass:pot}--\ref{ass:kernel} hold.
Let the initial data fulfill
\begin{align*}
	& \ph_0 \in H,
	\quad
	F(\ph_0) \in \Lx1,
	\quad
	(\ph_0)_\Omega \in [-1,1),
	\quad n_0 \in W,
	\\
	&
	0 \leq n_0(x) \leq 1 \quad \text{for a.e. } x \in \Omega ,
\end{align*}
and assume that $\h$ is \anlast{continuous and} uniformly bounded, $\chi$ is a nonnegative constant,
and $n_B \in L^\infty(Q)$ with $0 \leq n_B \leq 1$ a.e. in $Q$.
Then, problem \eqref{eq:1:abs:tumor}--\eqref{eq:4:abs:tumor} with $j=2$
admits a weak solution $(\ph, \mu, n)$ such that
\begin{align*}
	& \ph \in \H1 \Vp  \cap \L\infty H \cap \L2 V,
	\\
	& \ph \in L^\infty(Q), \quad |\ph(x,t)| <1 \quad\text{ for a.a. } (x,t) \in Q,
	\\
	& \mu \in \L p  V \cap \Ls2 V \quad \forall p\in (1,2) \quad  \forall \sigma \in (0,T),
    \\
    &\nabla \mu \in \L2 H,
    \\
    & n \in \H1 H \cap \L\infty V \cap \L2 W,
    \\
     & n \in L^\infty(Q), \quad 0 \leq n(x,t)\leq  1 \quad\text{ for a.a. } (x,t) \in Q,
\end{align*}
where
\begin{align*}
		\mu=  a \ph - K * \ph + F'(\ph) - \chi n, \quad & \text{$a.e.$ in $Q$},
\end{align*}
and the identities
\begin{align*}
	& \<\dt \ph ,v>_{V}
	+ \iO \nabla \mu \cdot \nabla v = \iO S(\ph,n) v,
	\\
	&\iO \dt n \, v
	+  \iO \nabla n \cdot \nabla v
	+ \iO \h(\ph) n v
	= {\cal B} \iO (n_B- n) v,
\end{align*}
are satisfied for every test function $ v\in V$, and almost everywhere in $(0,T)$.
Moreover, it holds that
\begin{align*}
	\ph(0)=\ph_0, \quad n(0)=n_0, \quad \text{a.e. in } \Omega.
\end{align*}
\end{theorem}

\begin{remark}
If \anlastr{the function $\h$ does not depend on the phase variable $\ph$}, then continuous dependence estimates can be proven
arguing along the lines of Theorem \ref{THM:UNI:LOCAL} and Theorem \ref{THM:UNI:NONLOCAL}.
In this case, both the local and the nonlocal problem are well posed.
\end{remark}

\begin{proof}
We just give a sketch of both the proofs since they are natural extensions of the previous ones with the choice $g = -\chi n$.

The difference is that here we do not know, a priori, whether $g= -\chi n $ fulfills \ref{ass:g} or not. However, a similar parabolic-type estimate as  \eqref{est:alpha:local} (\eqref{est:alpha:nonlocal} resp.) can be derived for $n$ and the strategy of the previous proof can be carried out with minor modifications.
Rigorously speaking, one should introduce a Faedo--Galerkin scheme, let $k$ denote the discretization parameter, for the system \eqref{eq:1:abs:tumor}--\eqref{eq:4:abs:tumor} along with the same $\l$-approximation introduced before.
Then, the discretized version of \eqref{eq:2:abs:tumor} can be tested by $n_k$. This gives
\begin{align*}
	\frac {1}2 \frac d{dt} \norma{n_k}^2 + \norma{\nabla n_k}^2 \leq C (\norma{n_k}^2 +1),
\end{align*}
so that
\begin{align*}
	\norma{n_k}_{\L\infty H \cap \L2 V} \leq C,
\end{align*}
whence, as before, we deduce
\begin{align*}
	\forall \,\alpha \in [0,1),
	\quad
	\exists\, C_\alpha>0:
	\quad
	\norma{n_k}_{L^\infty (0,t;H)}
	\leq
	C_\alpha t^{\alpha}
	\quad
	\forall\, t \in [0,T],
\end{align*}
with $C_\alpha>0$ that are independent of $k$.
Testing then by $\dt n_k$ and recalling that $n_0 \in V $, we find
\begin{align*}
	\norma{n_k}_{\H1  H \cap \L\infty V \cap \L2 W} \leq C,
\end{align*}
which, by the same argument, yields
\begin{align*}
	\forall \,\alpha \in [0,1),
	\quad
	\exists\, C_\alpha>0:
	\quad
	\norma{n_k}_{L^\infty (0,t;V)}
	\leq
	C_\alpha t^{\alpha},
	\quad
	\forall\, t \in [0,T],
\end{align*}
matching finally \ref{ass:g}.
The rest of the proof goes as in Section~\ref{existproof}.

Concerning the minimum and maximum principles for $n$, one can argue as done in \cite{SS}.
The final step consists in pass to the limit first as $k\to \infty$ and then as $\l \to 0$, along suitable subsequences.
\end{proof}

\section{Final comments and future work}
\label{final}
The present results are a first step towards a more general characterization of source terms which allow a pure initial datum.
Another interesting, and perhaps nontrivial, issue is the extension of our results to other models. For instance, Cahn--Hilliard--Navier--Stokes
system with sources (see \cite{HW} and references therein), fractional Cahn--Hilliard equations with sources (see, for instance, \cite{GGG2,GP} for the conserved case), or multi-component Cahn--Hilliard equations with sources (see \cite{DPG}, see also \cite{GGPS} for the general conserved case).

\appendix
\section{Appendix}

\begin{proof}[Proof of Proposition~\ref{prop:pot}]
Property~\ref{PROP:MZ:std} follows from \cite[Sec.~5]{GMS09} and \cite{MZ}.
Let us focus only on the proof of property \ref{PROP:MZ:sharp}.
Since $\widehat\beta$ is strictly convex with $\widehat\beta(0)=0$,
one has that $\beta$ is increasing, $\beta(0)=0$, with $\beta<0$ in $(-1,0)$,
and $\beta>0$ in $(0,1)$.

Let $\delta\in(0,1)$ be fixed, and let
$r\in(-1,1)$ and $r_0\in(-1,-1+\delta)$ as in property \ref{PROP:MZ:sharp}.
We define now
\[
  \underline r:=\frac{r_0-1}{2}=-1+\frac12(r_0+1),\qquad
  \overline r:=-1+\frac1\delta(r_0+1),
\]
so that it holds
\[
  -1<\underline r < r_0 <\overline r <0, \qquad
  \beta(\underline r), \beta(r_0), \beta(\overline r)<0.
\]
Let us consider the quantity
\[
  \beta(r)(r-r_0)
\]
and estimate it from below in the four cases $r\in(-1,\underline r)$,
$r\in[\underline r, \overline r]$, $r\in(\overline r,0]$, and $r\in(0,1)$.
\begin{itemize}
  \item $r\in(-1,\underline r)$: one has that
  \[
   \beta(r)(r-r_0)=  |\beta(r)||r-r_0|= (r_0-r)|\beta(r)|
   \geq(r_0-\underline r)|\beta(r)|
   =\frac{r_0+1}2|\beta(r)|.
  \]
  \item $r\in[\underline r, \overline r]$: one has that
  \begin{align*}
  \beta(r)(r-r_0)&\geq-|r-r_0|\max_{x\in[\underline r, \overline r]}|\beta(x)|
  \geq-(\overline r-\underline r)\max_{x\in[\underline r, \overline r]}|\beta(x)|\\
  &\geq\frac{r_0+1}2|\beta(r)| - \left[\frac{r_0+1}{2}+(\overline r-\underline r)\right]
  \max_{x\in[\underline r, \overline r]}|\beta(x)|\\
  &=\frac{r_0+1}2|\beta(r)|-
  \luca{\frac{1}{\delta}}(r_0+1)|\beta(\underline r)|.
  \end{align*}
  \item $r\in(\ov r, 0]$: one has that
  \begin{align*}
    \beta(r)(r-r_0)&\geq-|r-r_0||\beta(r)|=-(r-r_0)|\beta(r)|=r_0|\beta(r)| - r|\beta(r)|\\
    &=(r_0+1)|\beta(r)| - (r+1)|\beta(r)|\\
    &\geq\frac{r_0+1}2|\beta(r)|-(r+1)\max_{x\in[\overline r, 0]}|\beta(x)|\\
    &=\frac{r_0+1}2|\beta(r)| - (r+1)|\beta(\overline r)|
    \geq\frac{r_0+1}2|\beta(r)| - (r+1)|\beta(\underline r)|.
  \end{align*}
  \item $r\in(0,1)$: one has that
  \[
  \beta(r)(r-r_0)=|\beta(r)|(r-r_0)\geq-r_0|\beta(r)|\geq\frac{1-\delta}\delta(r_0+1)|\beta(r)|,
  \]
  where we have used the fact that $r_0\in(-1,-1+\delta)$ implies
  $-r_0>\frac{1-\delta}\delta(r_0+1)$.
\end{itemize}
Putting everything together, we obtain
\[
  \min\left\{\frac12, \frac{1-\delta}{\delta}\right\}(r_0+1)|\beta(r)|
  \leq \beta(r)(r-r_0)
  +\left[(r+1)
  +\luca{\frac{1}{\delta}}(r_0+1)\right]|\beta(\underline r)|
\]
and the thesis follows.
\end{proof}

\medskip
\noindent
{\bf Acknowledgements.} \mod{The authors thank the reviewers for their \anlast{valuable comments and suggestions}}.
The authors are members of Gruppo Nazionale per l'Analisi Ma\-te\-ma\-ti\-ca, la Probabilit\`{a} e le loro Applicazioni (GNAMPA), Istituto Nazionale di Alta Matematica (INdAM).
Besides, this research has been performed in the framework of the MIUR-PRIN Grant 2020F3NCPX ``Mathematics for industry 4.0 (Math4I4)".
AS gratefully acknowledge some support from the Alexander von Humboldt Foundation.
The present research has been supported by MUR, grant Dipartimento di Eccellenza 2023-2027.

\end{document}